\documentclass[11pt]{article}
\usepackage{amsmath, amssymb}
\usepackage{amsthm}
\usepackage{amstext}
\usepackage{appendix}
\usepackage{verbatim}
\usepackage{cases}
\newtheorem{theorem}{Theorem}[section]
\newtheorem{corollary}{Corollary}[section]
\newtheorem{proposition}{Proposition}[section]
\newtheorem{lemma}{Lemma}[section]

\newtheorem{conjecture}{Conjecture}[section]
\begin{document}

\title{Entire solutions of the magnetic Ginzburg-Landau equation in $\mathbb{R}^4$}
\author{Yong Liu\\Department of Mathematics, \\University of Science and Technology of China, Hefei, China,\\Email: yliumath@ustc.edu.cn
	\and 	Xinan Ma\\Department of Mathematics, \\University of Science and Technology of China, Hefei, China,\\Email: xinan@ustc.edu.cn
	\and Juncheng Wei\\Department of Mathematics, \\University of British Columbia, Vancouver, B.C., Canada, V6T 1Z2\\Email: jcwei@math.ubc.ca
	\and Wangze Wu\\Department of Mathematics, \\University of Science and Technology of China, Hefei, China,\\Email: wuwz18@mail.ustc.edu.cn}
\maketitle

\begin{abstract}
	We construct entire solutions of the magnetic Ginzburg-Landau equations in dimension 4 using Lyapunov-Schmidt reduction. The zero set of these solutions are close to the minimal submanifolds studied by Arezzo-Pacard\cite{Arezzo}. We also show the existence of a saddle type solution to the equations, whose zero set consists of two vertical planes in $\mathbb{R}^4$. These two types of solutions are believed to be energy minimizers of the corresponding energy functional and lie in the same connect component of the moduli space of entire solutions.
\end{abstract}
\section{Introduction}

In this paper, we consider the following system of magnetic Ginzburg-Landau equations:
\begin{equation}\label{Ginzburg Landau equations}
\begin{cases}
-\Delta_A\psi + \frac{\lambda}{2}(|\psi|^2 - 1)\psi = 0 \text{ in } \mathbb{R}^n,\\
d^*dA - Im(\nabla_A\psi\cdot\bar \psi) = 0 \text{ in } \mathbb{R}^n.
\end{cases}
\end{equation}
Here $\psi : \mathbb R^n \rightarrow \mathbb C$ is a complex valued function,  $A \in \Omega^1(\mathbb R^n, \mathbb R)$ is a one form, and $\lambda>0$ is a parameter.
We use $\nabla_A = \nabla - iA$ to denote the covariant gradient and $\Delta_A = \nabla_A\cdot\nabla_A$, $d^* = (-1)^{n(k+1)+1}*d*: \Omega^k(\mathbb R^n) \rightarrow \Omega^{k-1}(\mathbb R^n)$, where the operator $*:\Omega^k \rightarrow \Omega^{n- k}$ is the usual Hodge star operator on $\mathbb{R}^n$ with the standard metric.

The standard macroscopic theory of superconductivity is due to Ginzburg and Landau in \cite{Ginzburg} . The equations can be derived from the microscopic theory by the works of Bardeen, Cooper and Schrieffer \cite{Frank}, \cite{GLP} . The constant $\lambda > 0$ is depending on the material in question: when $\lambda < 1$, the material is of type I; when $\lambda > 1$, the material is of type II.
	
	In the case of dimension $n = 2$, the system \eqref{Ginzburg Landau equations} has a family of ``radial'' $j$-vortex solutions of the form $u_j = (\psi^{(j)}, A^{(j)})$, $j=1,...$, with
	\begin{equation}
		\psi^{(j)}(x) = U_j(r)e^{ij\phi}, \text{ and } A^{(j)} = V_j(r)\nabla(j\phi).
	\end{equation}
Here $(r, \phi)$ is the polar coordinate and $j =$ deg $\psi^{(j)}$ is an integer, and $U_j,V_j$ are real valued functions determined by a system of ODE. Moreover, both $U_j$ and  $ V_j$ tend to $1$ as $r\rightarrow \infty$  with exponential rate. Existence of this kind of solutions was proved in \cite{BC} using variational methods. The stability properties of these solutions were established in \cite{G}. More precisely, we now know that for $\lambda \le 1$, $u_j$ is stable; for $\lambda > 1$, $u_j$ is stable provided that $j = \pm 1$.
Using these $j$-vortex radial solutions as basic blocks, nonradial solutions in $\mathbb{R}^2$ of equation (\ref{Ginzburg Landau equations}) has been constructed in \cite{Ting1} using a finite dimensional Lyapunov-Schmidt reduction procedure. Using similar methods, solutions to the magnetic Ginzburg-Landau equation with external potential have been constructed in \cite{PTW}.

The magnetic Ginzburg-Landau equations (\ref{Ginzburg Landau equations}) is the Euler-Lagrange equations for the following energy functional:
\begin{equation}\label{energy functional}
E(A, \psi) = \int_{\mathbb R^n} |\nabla_A\psi|^2 + |dA|^2 + \frac{\lambda}{4}(1 - |\psi|^2)^2.
\end{equation}
Note that after a rescaling by a parameter $\epsilon$, the magnetic Ginzburg-Landau equations become
\begin{equation}\label{rescaled Ginzburg}
	\begin{cases}
		-\epsilon^2\Delta_A\psi + \frac{\lambda}{2}(|\psi|^2 - 1)\psi = 0,\\
		\epsilon^2d^*dA - Im(\nabla_A\psi\cdot\bar \psi) = 0.
	\end{cases}
\end{equation}
The corresponding energy functional is
\begin{equation*}
	E_\epsilon(A, \psi) =  \int_{\mathbb R^n}|\nabla_A\psi|^2 + \epsilon^2|dA|^2 + \frac{\lambda}{4\epsilon^2}(1 - |\psi|^2)^2.
\end{equation*}

For $\lambda=1$, it is called the self-dual case. In this case, the system can be reduced to first order equations. In this special case, it is proved in \cite{Alessandro} that if $(A_\epsilon, \psi_\epsilon)$ is a family of critical points for $E_\epsilon$ with a uniform finite energy bound, then as $\epsilon \rightarrow 0$, the energy measures
\begin{equation*}
	\mu_\epsilon:= \frac{1}{2\pi}[|\nabla_{A_\epsilon}\psi_\epsilon|^2 + \epsilon^2|dA_\epsilon|^2 + \frac{\lambda}{4\epsilon^2}(1 - |\psi_\epsilon|^2)^2]\mathcal H^n,
\end{equation*}
(where $\mathcal H^n$ means n-dimensional Hausdorff measures) converge subsequentially to the weight measure $\mu$ of a stationary, integral $(n - 2)$-varifold. In other words, when $\epsilon$ tends to $0$, the zero sets of $\psi_\epsilon$ are close to a codimension two minimal submanifold.
They also proved existence of varifold type solutions using variational methods.

In this paper, we shall investigate the general case of $\lambda>0$. We would like to construct some special smooth solutions from certain minimal submanifold with nice geometrical structure.  We will focus on the case of $n = 4$. Let us consider the following two planes:
\begin{equation*}
	\begin{aligned}
		\Pi_1 := \{(x_1, y_1, 0, 0) \in \mathbb R^4: x_1, y_1 \in \mathbb R\},\\
		\Pi_2 := \{( 0, 0, x_2, y_2) \in \mathbb R^4: x_2, y_2 \in \mathbb R\}.
	\end{aligned}
\end{equation*}
They intersect at the origin. Let $(r_1, \theta_1)$ and $(r_2, \theta_2)$ be the polar coordinates in $\Pi_1$ and $\Pi_2$, respectively. Our first result states that there exists a saddle type solution. This is the content of the following
\begin{theorem}\label{conclusion1}
	For $n = 4$, there exists a solution $(A, \psi)$ to \eqref{Ginzburg Landau equations}, satisfying
	\begin{equation*}
		\begin{aligned}
			&\psi = f(r_1, r_2) e^{i(\theta_1 + \theta_2)},\\
		&A = g(r_1, r_2)d\theta_1 + h(r_1, r_2)d\theta_2, \\
		\end{aligned}
	\end{equation*}
where $f, g, h$ are three real valued functions satisfying a system of second order PDE in the $(r_1,r_2)$ variables. The function $f$ vanish at the axes of the $r_1$-$r_2$ plane.
\end{theorem}
We shall call the solution stated in Theorem 1 as a saddle type solution. This is motivated by the saddle solution of the Allen-Cahn equation \begin{equation}\label{Allen-Cahn equation}
-\Delta u + u^3 - u = 0, \text{ in }\mathbb R^n.
\end{equation}
Indeed, for $n=2$, the Allen-Cahn equation has a solution $u$ which vanishes precisely at the two axes of the plane.

To proceed, let us recall the following minimal submanifold studied in Arezzo-Pacard\cite{Arezzo}. Let  $$\Gamma =\{(\rho e^{i\theta}, :\rho^{-1}e^{i\theta}): \rho,\theta \in \mathbb{R}\}.$$ Consider the scaled surface $\Gamma_\epsilon := \frac{1}{\epsilon}\Gamma$. As $\epsilon$ tends to $\infty$, formally $\Gamma_\epsilon$ tends to $\Pi_1 \cup \Pi_2$.  Our second result is the following theorem, where we use $\pm 1$ vortex solutions to construct solutions in $\mathbb{R}^4$.
\begin{theorem}\label{conclusion2}
	For $\epsilon>0$ small enough, there exists a sequence of solutions $(A_\epsilon, \psi_\epsilon)$ to \eqref{rescaled Ginzburg}. As $\epsilon \rightarrow 0$, whose zero sets of $\psi_\epsilon$ and $A_\epsilon$  converge to $\Gamma$ uniformly in the entire space.
\end{theorem}

We use infinite dimensional Lyapunov-Schmidt method to construct these solutions. More precise asymptotic behavior of these solutions will be provided in the subsequent sections.
We point out Brendle has done some related works for the self-dual magnetic Ginzburg-Landau equation in the unpublished paper\cite{Brendle}. Part of our arguments in this paper are inspired by his work. 

It is expected that the saddle type solution described in Theorem 1.1 and the solutions in Theorem 1.2 lie on the same connected component of the moduli space of entire solutions of the magnetic Ginzburg-Landau equation. This is also analogous to the case of Allen-Cahn equation. Indeed, for any even dimension not less than 8, the Allen-Cahn equation has a saddle type equation which vanishes on the Simons' cone. Moreover, there is family of solutions whose zero sets are asymptotic to the Simons' cone, See \cite{Cabre,Liu,LiuWang,Savin} for related results. These solutions are conjectured to be global minimizers of the corresponding Allen-Cahn energy functional, since the Simons' cone is area minimizing when the dimension $n\ge 8$. One also know that dimension 8 is the critical dimension. Sine the union of the two planes $\Pi_1$ and $\Pi_2$ are also area minimizing(See \cite{Lawlor}) in $\mathbb{R}^4$, we formulate the following
\begin{conjecture}
	The solutions of the magnetic Ginzburg-Landau equations provided in Theorem 1.1 and Theorem 1.2 are energy minimizers of the energy functional.
\end{conjecture}

This paper is organized as follows. In Section 2, we show the existence of saddle type solution and proved Theorem 1.1. Sections 3-8 will be devoted to the proof of Theorem 1.2. Section 3 studies the Fermi coordinate with respect to the minimal submanifold.  Section 4 deals with the Jacobi operator. In Section 5 we define the approximate solution and estimate its error. Section 6 is devoted to the linear theory of the linearized operator around the approximate solution. In Section 7, we study a projected nonlinear problem. Then in Section 8 we show the existence of true solutions and finish the proof of Theorem 1.2.  We put some of the tedious computations in Appendix.

\textbf{Acknowledgement } Y. Liu is partially supported by \textquotedblleft
The Fundamental Research Funds for the Central Universities
WK3470000014,\textquotedblright\ and NSFC no. 11971026. J. Wei is partially supported by NSERC of Canada.

\section{A saddle type solution whose zero set are two planes }
In this section, we show the existence of saddle type solution and prove Theorem \ref{conclusion1}. We seek for a solution to the magnetic Ginzburg Landau equations \eqref{Ginzburg Landau equations} with the form:
\begin{equation*}
	\begin{aligned}
		\psi &= f(r_1, r_2) e^{i(\theta_1 + \theta_2)},\\
		A &= g(r_1, r_2)d\theta_1 + h(r_1, r_2)d \theta_2 \\
		&= -\frac{g\sin\theta_1}{r_1}dx_1 + \frac{g\cos\theta_1}{r_1}dy_1 -\frac{h\sin\theta_2}{r_2}dx_2 + \frac{h\cos\theta_2}{r_2}dy_2,
	\end{aligned}
\end{equation*}
and $	\{ \psi = 0\} = \Pi_1 \cup \Pi_2.$

Let us denote $g_j := \partial_{r_j}g$, $h_j := \partial_{r_j}h$ and $f_j := \partial_{r_j}f$.

\begin{lemma}\label{rewritten equation in section 2}
	If $(A, \psi)$ is the solution to \eqref{Ginzburg Landau equations}, then $(f, g, h)$ satisfies
	\begin{numcases}{}
	\label{equation with respect to f}f_{11} + f_{22} + \frac{f_1}{r_1} + \frac{f_2}{r_2} - \frac{f(1 - g)^2}{ r_1^{2}} -\frac{ f(1 - h)^2}{r_2^{2}} - \frac{\lambda}{2}(f^2 - 1)f = 0,\\
		g_{11} + g_{22}- \frac{g_1}{r_1} + \frac{g_2}{r_2}  + f^2  - f^2g = 0, \\
		h_{11}+h_{22} - \frac{h_2}{r_2} + \frac{h_1}{r_1}  + f^2 - f^2h = 0,
\end{numcases}
with the energy functional
\begin{equation*}
	\begin{aligned}
		E(f, g, h) &= \frac{1}{2}\int_{\mathbb R^4} f_1^2 + f_2^2 + [(f - fg)^2 + g_1^2 + g_2^2]\frac{1}{r_1^2} + [(f - fh)^2 + h_1^2 + h_2^2]\frac{1}{r_2^2} + \frac{\lambda}{4}(1 - f^2)^2\\
		&= 2\pi^2\int_{r_1, r_2 > 0}\{|\nabla f|^2 + [(f - fg)^2 + |\nabla g|^2 ]\frac{1}{r_1^2} + [(f - fh)^2 + |\nabla h|^2 ]\frac{1}{r_2^2} \\
		&+ \frac{\lambda}{4}(1 - f^2)^2\}r_1r_2dr_1dr_2.
	\end{aligned}
\end{equation*}
\end{lemma}

\begin{proof}
	
	We compute
	\begin{equation*}
	\begin{aligned}
	\nabla_A\psi &= \nabla\psi - iA\psi \\
	&= \nabla f\cdot e^{i(\theta_1 + \theta_2)} + ie^{i(\theta_1 + \theta_2)}\nabla(\theta_1 + \theta_2)\cdot f - ie^{i(\theta_1 + \theta_2)} f\cdot( g\nabla\theta_1 + h\nabla \theta_2 ).\\
	\end{aligned}
	\end{equation*}
We also have	
	\begin{equation*}
	\begin{aligned}
	dA &= dg\wedge d\theta_1 + dh\wedge d\theta_2\\
	& = (\cos\theta_1\cdot g_1dx_1 + \sin\theta_1\cdot g_1dy_1 + \cos\theta_2\cdot g_2dx_2 + \sin\theta_2\cdot g_2dy_2 )\wedge d\theta_1\\
	& + (\cos\theta_1\cdot h_1dx_1 + \sin\theta_1\cdot h_1dy_1 + \cos\theta_2\cdot h_2dx_2 + \sin\theta_2\cdot h_2dy_2 )\wedge d\theta_2\\
	&= \frac{g_1}{r_1}dx_1\wedge dy_1 + (\frac{g_2\sin\theta_1\cos\theta_2}{r_1} - \frac{h_1\cos\theta_1\sin\theta_2}{r_2})dx_1\wedge dx_2\\
	&+ (\frac{g_2\sin\theta_1\sin\theta_2}{r_1} + \frac{h_1\cos\theta_1\cos\theta_2}{r_2})dx_1\wedge dy_2 + (-\frac{g_2\cos\theta_1\cos\theta_2}{r_1} - \frac{h_1\sin\theta_1\sin\theta_2}{r_2})dy_1\wedge dx_2\\
	& + (-\frac{g_2\cos\theta_1\sin\theta_2}{r_1} + \frac{h_1\sin\theta_1\cos\theta_2}{r_2})dy_1\wedge dy_2 + \frac{h_2}{r_2}dx_2\wedge dy_2.
	\end{aligned}
	\end{equation*}
Moreover,	
	\begin{equation*}
	\begin{aligned}
	&\Delta_A\psi = \Delta\psi + id^*A\cdot\psi - 2i<A, d\psi> - |A|^2\psi,\\
	&e^{-i(\theta_1 + \theta_2)}\Delta\psi = f_{11} + \frac{f_1}{r_1} - \frac{f}{r_1^2} + f_{22} + \frac{f_2}{r_2} - \frac{f}{r_2^2},\\
	&d^*A = \partial_{x_1}(\frac{g\sin\theta_1}{r_1}) - \partial_{y_1}(\frac{g\cos\theta_1}{r_1}) + \partial_{x_2}(\frac{h\sin\theta_2}{r_2}) - \partial_{y_2}(\frac{h\cos\theta_2}{r_2}) = 0,
	\end{aligned}
	\end{equation*}
	
	\begin{equation*}
	e^{-i(\theta_1 + \theta_2)}<A, d\psi> = i f(g r_1^{-2}  + hr_2^{-2}).
	\end{equation*}
It follows that	
	\begin{equation*}
	\begin{aligned}
	&e^{-i(\theta_1 + \theta_2)}\Delta_A\psi =e^{-i(\theta_1 + \theta_2)}( \Delta\psi + id^*A\cdot\psi - 2i<A, d\psi> - |A|^2\psi)\\
	&= f_{11} + \frac{f_1}{r_1} - \frac{f}{r_1^2} + f_{22} + \frac{f_2}{r_2} - \frac{f}{r_2^2} + 2f(gr_1^{-2}  + hr_2^{-2}) - f(g^2r_1^{-2}  + h^2r_2^{-2}).
	\end{aligned}
	\end{equation*}
	Let $d^*dA = (d^*dA)_1dx_1 + (d^*dA)_2dy_1 + (d^*dA)_3dx_2 + (d^*dA)_4dy_2$. Then there holds
	\begin{equation*}
	\begin{aligned}
	(d^*dA)_1 &= \partial_{y_1}(\frac{g_1}{r_1}) + \partial_{x_2}(-\frac{h_1\cos\theta_1\sin\theta_2}{r_2} + \frac{g_2\sin\theta_1\cos\theta_2}{r_1})\\
	&+ \partial_{y_2}(\frac{h_1\cos\theta_1\cos\theta_2}{r_2} + \frac{g_2\sin\theta_1\sin\theta_2}{r_1})\\
	&= (\frac{g_{11}}{r_1} - \frac{g_1}{r_1^2} + \frac{g_2}{r_1r_2} + \frac{g_{22}}{r_1})\sin\theta_1,
	\end{aligned}
	\end{equation*}

	\begin{equation*}
	\begin{aligned}
	(d^*dA)_2 &= -\partial_{x_1}(\frac{g_1}{r_1}) + \partial_{x_2}(-\frac{h_1\sin\theta_1\sin\theta_2}{r_2} - \frac{g_2\cos\theta_1\cos\theta_2}{r_1})\\
	&+ \partial_{y_2}(\frac{h_1\sin\theta_1\cos\theta_2}{r_2} - \frac{g_2\cos\theta_1\sin\theta_2}{r_1})\\
	& = -(\frac{g_{11}}{r_1} - \frac{g_1}{r_1^2} + \frac{g_2}{r_1r_2} + \frac{g_{22}}{r_1})\cos\theta_1,
	\end{aligned}
	\end{equation*}
	
	\begin{equation*}
	\begin{aligned}
	(d^*dA)_3 &= -\partial_{x_1}(-\frac{h_1\cos\theta_1\sin\theta_2}{r_2} + \frac{g_2\sin\theta_1\cos\theta_2}{r_1}) \\
	&- \partial_{y_1}(-\frac{h_1\sin\theta_1\sin\theta_2}{r_2} - \frac{g_2\cos\theta_1\cos\theta_2}{r_1}) + \partial_{y_2}(\frac{h_2}{r_2})\\
	&= (\frac{h_{22}}{r_2} - \frac{h_2}{r_2^2} + \frac{h_1}{r_1r_2} + \frac{h_{11}}{r_2})\sin\theta_2,
	\end{aligned}
	\end{equation*}
	
	\begin{equation*}
	\begin{aligned}
	(d^*dA)_4 &= -\partial_{x_1}(\frac{h_1\cos\theta_1\cos\theta_2}{r_2} + \frac{g_2\sin\theta_1\sin\theta_2}{r_1})\\
	&- \partial_{y_1}(\frac{h_1\sin\theta_1\cos\theta_2}{r_2} - \frac{g_2\cos\theta_1\sin\theta_2}{r_1}) - \partial_{x_2}(\frac{h_2}{r_2})\\
	&= -(\frac{h_{22}}{r_2} - \frac{h_2}{r_2^2} + \frac{h_1}{r_1r_2} + \frac{h_{11}}{r_2})\cos\theta_2.
	\end{aligned}
	\end{equation*}
Besides, we have:
	\begin{equation*}
	\begin{aligned}
	&Im(\nabla_A\psi\cdot\bar\psi) = (f^2 - f^2g)d\theta_1 + (f^2 - f^2h)d\theta_2\\
	&= (f^2 - f^2g)(-\frac{\sin\theta_1}{r_1}dx_1 + \frac{\cos\theta_1}{r_1}dy_1) + (f^2 - f^2h)(-\frac{\sin\theta_2}{r_2}dx_2 + \frac{\cos\theta_2}{r_2}dy_2).
	\end{aligned}
	\end{equation*}
Therefore, the energy functional for $f$ and $g,h$ has the form
	\begin{equation*}
	\begin{aligned}
	E(f, g, h) &= \frac{1}{2}\int_{\mathbb R^4} f_1^2 + f_2^2 + [(f - fg)^2 + g_1^2 + g_2^2]\frac{1}{r_1^2} + [(f - fh)^2 + h_1^2 + h_2^2]\frac{1}{r_2^2} + \frac{\lambda}{4}(1 - f^2)^2\\
	&= 2\pi^2\int_{r_1, r_2 > 0}\{|\nabla f|^2 + [(f - fg)^2 + |\nabla g|^2 ]\frac{1}{r_1^2} + [(f - fh)^2 + |\nabla h|^2 ]\frac{1}{r_2^2} \\
	&+ \frac{\lambda}{4}(1 - f^2)^2\}r_1r_2dr_1dr_2.
	\end{aligned}
	\end{equation*}
Then the equations \eqref{Ginzburg Landau equations} with respect to $f$ and $g,h$ in the form
	\begin{numcases}{}
	\label{equation with respect to f}f_{11} + f_{22} + \frac{f_1}{r_1} + \frac{f_2}{r_2} - \frac{f(1 - g)^2}{r_1^2} - \frac{f(1 - h)^2}{r_2^2} - \frac{\lambda}{2}(f^2 - 1)f = 0,\\
	g_{11} - \frac{g_1}{r_1} + \frac{g_2}{r_2} + g_{22} + f^2  - f^2g = 0, \\
	h_{22} - \frac{h_2}{r_2} + \frac{h_1}{r_1} + h_{11} + f^2 - f^2h = 0.
	\end{numcases}
	
\end{proof}

\begin{theorem}
	There exists a triple $(f, g, h)$, such that the equations in Lemma \ref{rewritten equation in section 2} hold, and satisfies the boundary condition:
	\begin{equation*}
		f =0 , \text{ on } \Pi_1 \cup \Pi_2.
	\end{equation*}
Moreover, $g(r_1,r_2)=h(r_2,r_1)$ and $g$ vanishes at the $r_2$ axis.
\end{theorem}

\begin{proof}
	We minimize the energy functional.
	
Although $r_1,r_2$ are polar coordinates and they are nonnegative, for convenience, during the proof we shall also consider the negative $r_1$ region by reflection. Let us consider the bounded region: $\Omega_k := \{(r_1, r_2):r_1 \in \mathbb{R}, r_2 > 0\} \cap B_k$, where $B_k$ is the ball of radius $k$ in the $r_1$-$r_2$ plane, centered at the origin.

Consider the corresponding energy functional
	
	\begin{equation*}
		\begin{aligned}
			&E_k(f, g, h) := \int_{\Omega_k}\{|\nabla f|^2 + [(f - fg)^2 + |\nabla g|^2 ]\frac{1}{r_1^2} + [(f - fh)^2 + |\nabla h|^2 ]\frac{1}{r_2^2} \\
		&+ \frac{\lambda}{4}(1 - f^2)^2\}r_1r_2dr_1dr_2.
		\end{aligned}
	\end{equation*}
Define the space\begin{align*}
H_{k}  & :=\{(f,g,h):f,g,h\in \tilde{H}_{0}^{1}(\Omega_{k});g(r_{1},r_{2}%
)=h(r_{2},r_{1}),\text{ }\\
&f(r_{1},r_{2})   =-f(-r_{1},r_{2}),g(r_{1},r_{2})=g(-r_{1},r_{2}),a.e.\}.
\end{align*}Here $\tilde{H}_{0}^{1}(\Omega_{k})$ is obtained by replacing the weight by $r_1r_2$ in the usual $H_{0}^{1}(\Omega_{k})$ space.  Let us consider the variational problem:
	\begin{equation*}
		\inf_{(f, g, h) \in H_k} E_k(f, g, h).
	\end{equation*}
The existence of a minimizer of  the action functional $E_k$ in this space follows from a straight forward modification of standard argument, by using the fact that the functional $E_k$ is coercive and satisfies the P.S. condition.

We can assume that the minimizer satisfies $0 \le f_k, g_k ,h_k\le 1$.  So using the interior estimates, we know $(f_k, g_k, h_k)$ converges to some $(f, g, h)$ in $H^1_{ loc}(\{r_1, r_2 > 0\})$. Hence $(f, g, h)$ is smooth in $\mathbb R^4 \backslash (\Pi_1 \cup\Pi_2)$.
	
	Next, we need to show the $(\psi, A)$ corresponding to $(f, g, h)$ is smooth in $\mathbb R^4$. In fact, let us first consider $(\psi_k, A_k)$ corresponding to $(f_k, g_k, h_k)$ and rewrite \eqref{Ginzburg Landau equations} as
	\begin{equation*}
		\begin{cases}
			-\Delta \psi = \frac{\lambda}{2}(1 - |\psi|^2)\psi + 2\psi(gr_1^{-2} + hr_2^{-2}) - \psi(g^2r_1^{-2} + h^2r_2^{-2}),\\
			-\Delta A= (f^2 - f^2g)d\theta_1 + (f^2 - f^2h)d\theta_2.
		\end{cases}
	\end{equation*}

	Note that $(f_k, g_k, h_k) \in H_{k+1}$ and $E_{k+1}(f_k, g_k, h_k) \ge E_{k+1}(f_{k+1}, g_{k+1}, h_{k+1})$, so we always have:
	\begin{equation*}
		E_k(f, g ,h) \le C_1,
	\end{equation*}
	where $C_1$ is a constant.
	
	Thus we know
	\begin{equation*}
		\begin{aligned}
			\|(f_k^2 - f_k^2g_k)d\theta_1\|_{L^2(\mathbb R^4)} \le \|(f_k - f_kg_k)r_1^{-1}\|_{L^2(\mathbb R^4)} \le \sqrt{C_1}.
		\end{aligned}
	\end{equation*}
	Then $\|A_k \|_{ W^{2, 2}(\mathbb R^4)} \le C_2$.
	
	Similarly, we can also get that
	\begin{equation*}
		\|\frac{\lambda}{2}(1 - |\psi_k|^2)\psi_k + 2\psi_k(gr_1^{-2} + hr_2^{-2}) - \psi_k(g^2r_1^{-2} + h^2r_2^{-2})\|_{L^2(\mathbb R^4)} \le C_3.
	\end{equation*}		
	Then $\|\psi_k \|_{ W^{2, 2}(\mathbb R^4)} \le C_4$.
	
	Therefore, $(A_k, \psi_k)$ converges to some $(A, \psi)$ in $W^{1,2}(\mathbb R^4)$ norm. Hence $(A, \psi)$ is one solution to \eqref{Ginzburg Landau equations} and it is smooth.
	By minimum principle, we know in $\mathbb R^4\backslash(\Pi_1 \cup \Pi_2)$, $f > 0$.
 \end{proof}

\section {Fermi coordinates}
Instead of working in the usual Euclidean coordinate, we shall work in the coordinate adapted to the minimal submanifold, which we call Fermi coordinate here. Let us recall the minimal submanifold $\Gamma $ studied in \cite{Arezzo}:
$\Gamma  =   \left( \frac{\sqrt2}{2}\rho e^{i\theta }, \frac{\sqrt2}{2}\rho^{-1} e^{i\theta }   \right)$.

We rewrite the minimal submanifold $\Gamma$ as:
\begin{equation*}
	\Gamma = \frac{e^{is}}{\sqrt{\sin2s}}\Theta,\\
\end{equation*}
where $s \in (0, \frac{\pi}{2})$, $\Theta = (\cos\theta, \sin\theta)\in S^1$.
The rescaled manifold $\Gamma_\epsilon$ is given by
\begin{equation*}
	 (\frac{\sqrt2}{2\epsilon}\rho e^{i\theta}, \frac{\sqrt2}{2\epsilon}\rho^{-1}e^{i\theta}),
\end{equation*}
The tangent vector of $\Gamma$ can be written in the form

\begin{equation*}
		\partial_s \Gamma = \frac{-e^{-is}}{(\sin2s)^{\frac32}}\Theta, \text{ }
		\partial_\theta\Gamma = \frac{e^{is}}{\sqrt{\sin2s}}\Theta^{\perp},
\end{equation*}
where $\Theta^{\perp} = (-\sin\theta, \cos\theta)$. We also have the
unit tangent vector $e_1 = e^{-is}\Theta$, $e_2 = e^{is}\Theta^\perp$, and unit normal vector $m = ie^{-is}\Theta$, $n = ie^{is}\Theta^\perp$.

Define a map $T: (s, \theta, a, b) \longrightarrow (z_1, z_2) \in \mathbb C^2$:
\begin{equation*}
	\begin{aligned}
		Y &= \Gamma_\epsilon + am + bn\\
		&= \left( \frac{\cos s}{\epsilon\sqrt{\sin2s}}\Theta + a\sin s\cdot\Theta - b\sin s\cdot\Theta^\perp, \frac{\sin s}{\epsilon\sqrt{\sin2s}}\Theta + a\cos s\cdot\Theta + b\cos s\cdot\Theta^\perp\right).
	\end{aligned}
\end{equation*}
Denote $\Sigma_\epsilon : = \{(s, \theta, a, b): a^2 + b^2 < \frac{1}{\epsilon^2\sin2s}\}$.
We would like to show that if $(s,\theta, a, b) \in \Sigma_\epsilon$, then $Y = \Gamma_\epsilon + am + bn$ defined a smooth coordinate system around the minimal submanifold.   To see this,  we need several lemmas.
\begin{lemma}\label{lemma1}
Let $P = \left(\rho _1 e^{i\theta _1}, \rho _2 e^{i\theta _2}\right
)$, $ \rho_1 \neq \rho_2$. If there are two points $P_1$, $P_2 \in \Gamma$ satisfying    $dist(P_j,P) = dist(\Gamma_\epsilon, P)$, then $P_1 = P_2$.

\end{lemma}

\begin{proof}
The distance between $P$ and $P_1$ is
	\begin{equation*}
	\begin{aligned}
	d^2 &= \inf _{\rho, \theta} \frac{1}{2} \rho^2 + \frac12\rho^{-2} + \rho_1^2  + \rho_2^2 - \sqrt2\rho _1\rho\cos(\theta - \theta _1) - \sqrt2\rho_2\rho^{-1}\cos(\theta - \theta_2)\\
	&= \inf_\rho \frac12\rho^2 + \frac12\rho^{-2} + \rho_1^2 + \rho_2^2 - \sqrt2\sqrt{\rho_1^2\rho^2 + \rho_2^2\rho^{-2} + 2\rho_1\rho_2\cos(\theta_1 - \theta_2)}.
	\end{aligned}
	\end{equation*}
Set $f(\rho) = \frac12\rho^2 + \frac12\rho^{-2}  - \sqrt2\sqrt{\rho_1^2\rho^2 + \rho_2^2\rho^{-2} + 2\rho_1\rho_2\cos(\theta_1 - \theta_2)} $.
	
	Suppose $\rho_1 > \rho_2$ .
	We observe that if $\rho \geqslant 1$, then
	\begin{equation*}
	f(\rho) \leqslant f(\rho^{-1}),
	\end{equation*}
Moreover, ``=" holds iff $\rho$ = 1. So $\inf\limits_{\rho > 0} f(\rho) = \inf\limits_{\rho \geqslant 1} f(\rho)$.
	
	Suppose there exist $r_1, r_2 \geqslant 1$, satisfying
	\begin{equation*}
	f(r_1) = f(r_2) = \inf_{\rho > 0} f(\rho),
	\end{equation*}
then $f'(r_1) = f'(r_2) = 0$. Let us define $\alpha = \theta_1 - \theta_2$, we have
	\begin{equation*}
	\begin{aligned}
	f'(\rho) &= \rho - \rho^{-3} - \sqrt2\frac{\rho_1^2\rho - \rho_2^2\rho^{-3}} {\sqrt{\rho_1^2\rho^2 + \rho_2^2\rho^{-2} + 2\rho_1\rho_2\cos(\alpha)}},\\
	f'(1) &= \frac{-\sqrt2(\rho_1^2 - \rho_2^2)} {\sqrt{(\rho_1 - \rho_2)^2 + 2\rho_1\rho_2(1 + \cos\alpha)}}\\
	&> \frac{-\sqrt2(\rho_1^2 - \rho_2^2)}{\rho_1 - \rho_2}\\
	&= -\sqrt2(\rho_1 + \rho_2) < 0.
	\end{aligned}
	\end{equation*}
	So $\rho$ = 1 is not an extreme point. As a result, we obtain $r_1 > r_2 >1$.
	If $f'(\rho) = 0$, then
	\begin{equation*}
	\begin{aligned}
	f(\rho) &= \frac12\rho^2 + \frac12\rho^{-2} -2\frac{\rho_1^2\rho - \rho_2^2\rho^{-3}}{\rho - \rho^{-3}}\\
	&=\frac{\rho^6 - \rho^{-2} - 4\rho_1^2\rho^4 + 4\rho_2^2}{2(\rho^4 - 1)} =: g(\rho).
	\end{aligned}
	\end{equation*}
	Since
	\begin{equation*}
	g'(\rho) = \frac{2(\rho^3 - \rho^{-1})^3 + 16\rho^3(\rho_1^2 - \rho_2^2)}{2(\rho^4 - 1)^2} > 0,
	\end{equation*}
	$g$ is strictly increasing.
	If $\ g(\rho_1) = g(\rho_2)$, we have $r_1 = r_2$, which is a contradiction. So when $\rho_1 \neq \rho_2$, the closest point is unique, which completes the proof.
\end{proof}

The proof of the next three lemmas relies on careful analysis of the geomeotric structure around the minimal submanifold. The computation is straight forward but tedious, we put the proof in the Appendix.
\begin{lemma}\label{lemma2}
Suppose $\rho_1 > 0$, then $(\rho_1e^{i\theta_1}, \rho_1e^{i\theta_2})$ has two different closest points to $\Gamma$ if and only if $\rho_1 > \sqrt{2 + 2\cos\alpha}$.

\end{lemma}

Next, we want to show what the conditions on a, b are needed, if $y = (\rho_1 e^{i\theta_1}, \rho_2e^{i\theta_2})\in Y$ has a unique closest point to $\Gamma_\epsilon$.

\begin{lemma}\label{lemma3}
	If $y = (\rho_1 e^{i\theta_1}, \rho_2e^{i\theta_2}) = T(s, \theta, a, b)$, then y has a unique closest point to $\Gamma_\epsilon$ if and only if $a^2 + b^2 < \frac1{\epsilon^2\sin2s}$.
\label{3}
\end{lemma}

\begin{lemma}\label{lemma4}
If 	$y = T(s, \theta, a, b)$, and $a^2 + b^2 < \frac1{\epsilon^2\sin2s}$, then dist(y, $\Gamma_\epsilon$) = $\sqrt{a^2 + b^2}$.
\label{4}
\end{lemma}

>From previous arguments, we know that if
\begin{equation*}
	\begin{aligned}
       y &= T(s, \theta, a, b)\\
       &= \left( \frac{\cos s}{\epsilon\sqrt{\sin2s}}\Theta + a\sin s\cdot\Theta - b\sin s\cdot\Theta^\perp, \frac{\sin s}{\epsilon\sqrt{\sin2s}}\Theta + a\cos s\cdot\Theta + b\cos s\cdot\Theta^\perp\right) \\
       &\in T(\Sigma_\epsilon).
    \end{aligned}
\end{equation*}
We can define a map S: $T(\Sigma_\epsilon) \rightarrow \Sigma_\epsilon$ as follows: if $P \in T(\Sigma_\epsilon)$, then we suppose $P = T(s_0, \theta_0, a_0, b_0)$ and $Q_0 = \frac{e^{is_0}}{\epsilon\sqrt{\sin2s_0}}\Theta_0 \in \Gamma_\epsilon$, satisfying $a_0^2 + b_0^2 < \frac1{\epsilon^2\sin2s_0}$. By Lemma \ref{lemma4}, $Q_0$ is the closest point. But we know $P \in T(\Sigma_\epsilon)$, so the closest point is unique. As a result, let $S(P) = (s_0, \theta_0, a_0, b_0)$,
\begin{equation*}
   \begin{aligned}
	  &a = (P - Q_0)\cdot m_0,\\
	  &b = (P - Q_0)\cdot n_0,
   \end{aligned}
\end{equation*}
where $m_0, n_0$ are the normal vector at Q. Then T is well-defined and $T\circ S = Id$, $S\circ T = Id$. So $Y = \Gamma_\epsilon + am + bn$ is a coordinate in $\Sigma_\epsilon$.

Denote $a = r\cos\phi$, $b = r\sin\phi$. By the definition, we know $r$ and $\phi$ are smooth in $\Sigma_\epsilon$. Finally in this section, we show the regularity of $r$ and $\phi$ in $\mathbb R^4$.
\begin{lemma}
 $\phi$ is Lipschitz continuous near $\mathbb C^2\backslash \Sigma_\epsilon$ and r is Lipschitz continuous in $\mathbb C^2$. Especially $\phi$ is $C^{1,1}$ near 0.

\end{lemma}

\begin{proof}
As $\vert\nabla r\vert \leq 1$, r is Lipschitz continuous, but not $C^1$. Because when $x \in \mathbb C^2 \backslash \Sigma_\epsilon$, the nearest point is not unique. So r is continuous but not smooth.
Next we will show that $\phi$ is continuous.

Consider $\Gamma = \frac{e^{is}}{\sqrt{\sin2s}}\Theta$. Denote $P = (\rho_1e^{i\theta_1}, \rho_2e^{i\theta_2}) \neq 0$. Suppose $Q_1$ is a closest point to $P$:
\begin{equation*}
	\begin{aligned}
		Q_1 = (\frac{\sqrt2}2\rho e^{i\theta}, \frac{\sqrt2}2\rho^{-1}e^{i\theta}).\\
	\end{aligned}
\end{equation*}
Then the normal vectors at $P$ are:
\begin{equation*}
	\begin{aligned}
		&m_1 = (\sin s\cdot e^{i\theta}, \cos s\cdot e^{i\theta}),\\
		&n_1 = (-\sin s\cdot ie^{i\theta}, \cos s\cdot ie^{i\theta}),\\
		&\frac{\sqrt2}2\rho = \frac{\cos s}{\sqrt{\sin2s}}\Rightarrow  \rho^2 = \frac{\cos s}{\sin s}.
	\end{aligned}
\end{equation*}
>From previous arguments, we get
\begin{equation*}
	\begin{aligned}
		\cos\theta = \frac{\rho_1\rho\cos\theta_1 + \rho_2\rho^{-1}\cos\theta_2}{\sqrt{\rho_1^2\rho^2 + \rho_2^2\rho^{-2} + 2\rho_1\rho_2\cos(\theta_1 - \theta_2)}}, \\
		\sin\theta = \frac{\rho_1\rho\sin\theta_1 + \rho_2\rho^{-1}\sin\theta_2}{\sqrt{\rho_1^2\rho^2 + \rho_2^2\rho^{-2} + 2\rho_1\rho_2\cos(\theta_1 - \theta_2)}}, \\
		\cos(\theta - \theta_1) = \frac{\rho_1\rho + \rho_2\rho^{-1}\cos(\theta_2 - \theta_1)}{\sqrt{\rho_1^2\rho^2 + \rho_2^2\rho^{-2} + 2\rho_1\rho_2\cos(\theta_1 - \theta_2)}},\\
		\sin(\theta - \theta_1) = \frac{\rho_2\rho^{-1}\sin(\theta_2 - \theta_1)}{\sqrt{\rho_1^2\rho^2 + \rho_2^2\rho^{-2} + 2\rho_1\rho_2\cos(\theta_1 - \theta_2)}},\\
		\cos(\theta - \theta_2) = \frac{\rho_2\rho^{-1} + \rho_1\rho\cos(\theta_2 - \theta_1)}{\sqrt{\rho_1^2\rho^2 + \rho_2^2\rho^{-2} + 2\rho_1\rho_2\cos(\theta_1 - \theta_2)}},\\
		\sin(\theta - \theta_2) = \frac{-\rho_1\rho\sin(\theta_2 - \theta_1)}{\sqrt{\rho_1^2\rho^2 + \rho_2^2\rho^{-2} + 2\rho_1\rho_2\cos(\theta_1 - \theta_2)}},\\
		\rho\sin s = \sqrt{\frac{\cos s}{\sin s}}\cdot\sin s = \sqrt{\frac{\sin 2s}2},\\
		\rho^{-1}\cos s = \sqrt{\frac{\sin 2s}2},
	\end{aligned}
\end{equation*}

\begin{equation*}
	\begin{aligned}
		a_1 &= (P - Q_1)\cdot m\\
		&= \rho_1 \sin s\cdot\frac{\rho_1\rho + \rho_2\rho^{-1}\cos(\theta_2 - \theta_1)}{\sqrt{\rho_1^2\rho^2 + \rho_2^2\rho^{-2} + 2\rho_1\rho_2\cos(\theta_1 - \theta_2)}} - \frac{\sqrt2}2 \rho\sin s +\\
		&\rho_2\cos s\cdot\frac{\rho_2\rho^{-1} + \rho_1\rho\cos(\theta_2 - \theta_1)}{\sqrt{\rho_1^2\rho^2 + \rho_2^2\rho^{-2} + 2\rho_1\rho_2\cos(\theta_1 - \theta_2)}} - \frac{\sqrt2}2\rho^{-1}\cos s,\\
		b_1 &= \rho_1\sin s\cdot \frac{\rho_2\rho^{-1}\sin(\theta_2 - \theta_1)}{\sqrt{\rho_1^2\rho^2 + \rho_2^2\rho^{-2} + 2\rho_1\rho_2\cos(\theta_1 - \theta_2)}} + \\
		&\rho_2\cos s\cdot \frac{\rho_1\rho\sin(\theta_2 - \theta_1)}{\sqrt{\rho_1^2\rho^2 + \rho_2^2\rho^{-2} + 2\rho_1\rho_2\cos(\theta_1 - \theta_2)}}	.
		\end{aligned}
\end{equation*}

If $\rho_1 = \rho_2 > 0$, then there are two closest points to $P$, one of which is $Q_1$. Let the other be $Q_2$. By symmetry, we know:
\begin{equation*}
	Q_2 = (\frac{\sqrt2}2\rho^{-1}e^{i\theta}, \frac{\sqrt2}2\rho e^{i\theta}).
\end{equation*}

 We observe that, when $Q_1 \rightarrow Q_2$, we have $\rho \rightarrow \rho^{-1}$, $s \rightarrow (\frac\pi2 - s)$.
So $a_1 = a_2$, $b_1 = b_2$.

If $\rho_1 \neq \rho_2$,  let $\rho_1, \rho_2$ tend to $0$. As a result, $\rho$ will tend to 1, and
\begin{equation*}
	\begin{aligned}
		&a_1 = -1,\\
		&b_1 = 0.
	\end{aligned}
\end{equation*}
So $\phi(0) = \pi$. Therefore $\phi$ is continuous around $\mathbb C^2\backslash \Sigma_\epsilon$. So it is Lipschitz continuous around $\mathbb C^2\backslash \Sigma_\epsilon$. At the origin, we consider two directions: $(\rho_1, \theta_1, \rho_2, \theta_2)$, $(\rho_1, \theta_1 + \pi, \rho_2, \theta_2 + \pi)$. As $\rho$ is Lipschitz continuous near 0, we deduce $\nabla \phi(0) = 0$. So $\phi$ is $C^{1,1}$ near the origin.
\end{proof}
We remark that from the expression of $a_1$, $b_1$, we deduce that
\begin{equation*}
	\begin{aligned}
		\phi (\rho_1, \theta_1, \rho_2,\theta_2) = -\phi(\rho_1, \theta_2, \rho_2, \theta_1) = \phi(\rho_2, \theta_1, \rho_1, \theta_2) = -\phi(\rho_1, -\theta_1, \rho_2, -\theta_2);\\
		r(\rho_1, \theta_1, \rho_2, \theta_2) = r(\rho_2, \theta_1, \rho_1, \theta_2) = r(\rho_1, \theta_2, \rho_2, \theta_1) = r(\rho_1, -\theta_1, \rho_2, -\theta_2).\\
	\end{aligned}
\end{equation*}

\section{ The Jacobi operator}
In this section, we show that the Jacobi operator naturally arise from the projection of the error of a naturally defined approximate solution around the manifold. We recall that the energy functional in $\mathbb R^4$ is:
\begin{equation*}
	E(\psi, A) = \frac12\int _{\mathbb R^4}\{|\nabla_A\psi|^2 + |d A|^2 + \frac{\lambda}{4}(|\psi|^2-1)^2\},
\end{equation*}
for the fields
\begin{equation*}
	\begin{aligned}
		&A = (A_1, A_2, A_3, A_4):\mathbb R^4 \rightarrow \mathbb R^4,\\
		&\psi = \psi_1 + i\psi_2: \mathbb R^4\rightarrow \mathbb C.
	\end{aligned}
\end{equation*}
And $\nabla_A = \nabla - iA$ is the covariant gradient. As we can consider $A$ as a 1-form on $\mathbb R^4$, then $dA$ is well-defined.

\begin{equation*}
	\begin{aligned}
		|\nabla_A\psi|^2 = \Sigma_{j=1}^4|\partial_j\psi - iA_j\psi|^2,\\
		|d A|^2 = |\partial_jA_k - \partial_kA_j|^2.
	\end{aligned}
\end{equation*}

Critical point $u = (\psi, A)$ of $E(\psi, A)$ satisfies the Ginzburg-Landau (GL) equations
\begin{equation}
	-\Delta_A\psi + \frac{\lambda}{2}(|\psi|^2-1)\psi = 0,
\end{equation}
\begin{equation}\label{GL2}
	\begin{aligned}
		\Sigma_{k\neq j}\partial_k(\partial_kA_j - \partial_jA_k) + Im((\nabla_{A})_{e_j}\psi\cdot\bar{\psi}) = 0, j = 1,2,3,4.
	\end{aligned}
\end{equation}
Equations \eqref{GL2} has the form
\begin{equation}
	d^*dA = Im(\nabla_A\psi\cdot\bar{\psi}).
\end{equation}
We also have
\begin{equation*}
	\begin{aligned}
		\Delta_A \psi &= \nabla_A\nabla_A \psi \\
		&= \Delta\psi - i\partial_jA_j\psi - 2iA_j\partial_j\psi - |A|^2\psi\\
		&=\Delta\psi + id^*A\psi - 2i<A, d\psi>-|A|^2\psi.
	\end{aligned}
\end{equation*}
Let us define
\begin{equation*}
	F(u) =
	\begin{pmatrix}
		-\Delta_A\psi + \frac{\lambda}{2}(|\psi|^2-1)\psi\\
		d^*dA - Im(\nabla_A\psi\cdot\bar{\psi})
	\end{pmatrix}.
\end{equation*}

Let $ W := (A^{(1)}, \psi^{(1)}) = (V(r)\nabla \phi, U(r)e^{i\phi})$ be the +1 vortex solution. Then we have:

	\begin{equation*}
		\begin{aligned}
			A_3 &= -\frac{V}{r}\sin\phi, \text{ }
			A_4 = \frac{V}{r}\cos\phi,\\
			\partial_3A_3 &= \frac{V}{r^2}\sin2\phi - \frac{V'}{2r}\sin2\phi,\\
			\partial_4A_3 &= -\frac{V'}{r}\sin^2\phi - \frac{V}{r^2}\cos2\phi,\\
			\partial_3A_4 &= \frac{V'}{r}\cos^2\phi - \frac{V}{r^2}\cos2\phi,\\
			\partial_4A_4 &= \frac{V'}{2r}\sin2\phi - \frac{V}{r^2}\sin2\phi.\\
		\end{aligned}
	\end{equation*}		

In the following context, we use the following properties(see \cite{BC}, \cite{P} ):
	\begin{equation*}
		\begin{aligned}
			1.&-V'' + \frac{V'}{r} - U^2 + VU^2 = 0,\\
			2.&U = cr, V = dr^2, \text{as} \  r \rightarrow 0,\\
			3.&U = 1+O(e^{-m_\lambda r}), V = 1+ O(e^{-r}) \ \text{as}\ r\rightarrow \infty, \text{with } m_\lambda := \min{(\sqrt\lambda, 2)}.
		\end{aligned}
	\end{equation*}

Let $\epsilon$ be a small parameter and denote $t_1 = a - \epsilon f$, $t_2= b - \epsilon g$, $t_1 + it_2 = \tilde re^{i\tilde \phi} $. When it is close to $\Gamma_\epsilon$, we define the following $(\psi, A)$ as the approximation solution:
\begin{equation*}
	\begin{aligned}
		\psi(a, b) &:= \psi^{(1)}(a - \epsilon f, b - \epsilon g),\\
		A &= A_j dy^j\\
		& :=V(\tilde r)\nabla(\phi(a - \epsilon f, b - \epsilon g))\\
		&= (-\epsilon f_sA_3 - \epsilon g_sA_4)dy^1 + (-\epsilon f_\theta A_3 - \epsilon g_\theta A_4)dy^2\\
		&+A_3dy^3 + A_4dy^4\\
		&= (\epsilon f_s\cdot \frac{V}{\tilde r}\sin\tilde\phi - \epsilon g_s\cdot\frac{V}{\tilde r}\cos\tilde\phi)dy^1\\
		&+ (\epsilon f_\theta\cdot \frac{V}{\tilde r}\sin\tilde\phi - \epsilon g_\theta\cdot\frac{V}{\tilde r}\cos\tilde\phi)dy^2\\
		& -\frac{V}{\tilde r}\sin\tilde\phi dy^3 + \frac{V}{\tilde r}\cos\tilde\phi dy^4.\\
	\end{aligned}
\end{equation*}
	
	Denote $y = (\frac{\cos s}{\sqrt{\sin 2s}}\Theta + a \sin s\cdot\Theta - b\sin s\cdot\Theta^\perp, \frac{\sin s}{\sqrt{\sin2s}}\Theta + a\cos s\cdot\Theta + b\cos s\cdot\Theta^\perp)$, and $\Gamma_\epsilon = \frac{1}{\epsilon}\Gamma = (\frac{\cos s}{\epsilon\sqrt{\sin 2s}}\Theta, \frac{\sin2}{\epsilon\sqrt{\sin2s}}\Theta)$. Then we can get a coordinate with respect to $(s, \theta, a, b)$ near $\Gamma_\epsilon$. Denote $\rho = (\sin2s)^{-\frac12}$ for simplicity.
	
	Recall the Jacobian operator of $\Gamma$ (see \cite{Arezzo}):
For normal vector field $N = ie^{-is}f\cdot\Theta + ie^{is}g\Theta^\perp = fm + gn$, we have
\begin{equation}\label{Jacobi operator}
	\begin{aligned}
		L_HN &= ie^{-is}\Theta[(\sin2s)^3\partial_s^2f + 2(\sin2s)^2\cos2s\cdot\partial_s f + \sin2s\cdot\partial_\theta^2f -\\
		& 2\sin2s\cos2s\cdot\partial_\theta g + 2(\sin2s)^3\cdot f - \sin2s(\cos2s)^2\cdot f] + ie^{is}\Theta^\perp \cdot \\
		&[(\sin2s)^3\partial_s^2g + 2(\sin2s)^2\cos2s\cdot\partial_s g + \sin2s\cdot\partial_\theta^2g + 2\sin2s\cos2s\cdot\partial_\theta f +\\
		&2(\sin2s)^3\cdot g - \sin2s(\cos2s)^2\cdot g].
	\end{aligned}
\end{equation}
	
	In this section, we show
	\begin{theorem}\label{Jacobi}
		When $\epsilon$ is small enough, the equations
		\begin{equation*}
			\begin{cases}
				<F(u), (\partial_3\psi-iA_3\psi, -\partial_3 A + \nabla A_3)> = 0,\\
				<F(u), (\partial_4\psi-iA_4\psi, -\partial_4 A + \nabla A_4)> = 0,\\
			\end{cases}
		\end{equation*}
can be written as
		\begin{equation*}
			L_H(f, g)  = (G_1, G_2)
		\end{equation*}
		where the inner product is the integral over the whole space $(a, b) \in \mathbb R^2$. $G_1,G_2=O(\epsilon)$ is of higher order in $f,g$
	\end{theorem}

	Before we prove Theorem \ref{Jacobi}, we first give several lemmas, whose proof can be found in the Appendix, since again they follows from tedious but routine computations.
	
	Denote $dA = \Sigma_{j<k} A_{jk}dy^j\wedge dy^k$, then we have:
	\begin{equation*}
		\begin{aligned}
			A_{jk} &= \partial_j A_k - \partial_k A_j,\\
			A_{12} &= O(\epsilon^2),\\
			A_{13} &= \epsilon g_s(\partial_3A_4 - \partial_4A_3),\\
			A_{14} &= -\epsilon f_s(\partial_3 A_4 - \partial_4A_3),\\
			A_{23} & = \epsilon g_\theta(\partial_3A_4 - \partial_4A_3),\\
			A_{24} &= -\epsilon f_\theta(\partial_3A_4 - \partial_4A_3),\\
			A_{34} &= \partial_3A_4 - \partial_4A_3 = \frac{V'}{r}. \\
		\end{aligned}
	\end{equation*}
	
	\begin{lemma}\label{lemma4.1}
	There holds
		\begin{equation*}
		\begin{aligned}
			 &-\Delta\psi = 2\epsilon^2(\sin2s)^2\cos2s(\epsilon f_s\partial_3\psi + \epsilon g_s \partial_4\psi ) + \epsilon^2(\sin2s)^3(\epsilon f_{ss}\partial_3\psi + \epsilon g_{ss} \partial_4\psi ) + \\
		&\epsilon^2\sin2s(\epsilon f_{\theta\theta}\partial_3\psi + \epsilon g_{\theta\theta}\partial_4\psi) + 2b\epsilon^2\sin2s\cos2s(\epsilon f_\theta \partial_3\partial_3\psi + \epsilon g_{\theta}\partial_3\partial_4\psi) + \\
		&2a\epsilon^2\sin2s\cos2s(-\epsilon f_\theta \partial_3\partial_4\psi - \epsilon g_\theta \partial_4\partial_4\psi) + 2a\epsilon^2(\sin2s)^3\partial_3\psi - \partial_3\partial_3\psi - \\
		&b^2\epsilon^2\sin2s(\cos2s)^2\partial_3\partial_3\psi + a\epsilon^2\sin2s(\cos2s)^2\partial_3\psi + 2ab\epsilon^2\sin2s(\cos2s)^2\partial_3\partial_4\psi +\\
		&b\epsilon^2\sin2s(\cos2s)^2\partial_4\psi + 2b\epsilon^2(\sin2s)^3\partial_4\psi - \partial_4\partial_4\psi - a^2\epsilon^2\sin2s(\cos2s)^2\partial_4\partial_4\psi  + \\
		&O(\epsilon^{4}).\\
		\end{aligned}
	\end{equation*}
	\end{lemma}
	
	\begin{lemma}\label{lemma4.2}
	We have:
		\begin{equation*}
		\begin{aligned}
			d^*A &= \epsilon^2\rho^{-6}\cdot(\epsilon f_{ss}\cdot A_3 + \epsilon g_{ss}\cdot A_4) + \epsilon^2\rho^{-2}\cdot(\epsilon f_{\theta\theta} A_3 + \epsilon g_{\theta\theta}A_4) \\
			&+2b\cos2s\cdot\epsilon^3\rho^{-2}f_\theta\partial_3A_3 + b\cos2s\cdot\epsilon^3\rho^{-2}g_\theta(\partial_4A_3 +\partial_3A_4) \\
			&-a\cos2s\cdot\epsilon^3\rho^{-2}f_\theta(\partial_3A_4 +\partial_4A_3) - 2a\cos2s\cdot\epsilon^3\rho^{-2}g_\theta\partial_4A_4\\
			&-(-ab(\cos2s)^2\epsilon^2\rho^{-2} + 2a^2b(\cos2s)^2\epsilon^3\rho^{-5})(\partial_4A_3+\partial
			_3A_4)\\
			& -(1+b^2(\cos2s)^2\epsilon^2\rho^{-2} - 2ab^2(\cos2s)^2\epsilon^3\rho^{-5})\partial_3A_3\\
			& - (1+a^2(\cos2s)^2\epsilon^2\rho^{-2}-2a^3(\cos2s)^2\epsilon^3\rho^{-5})\partial_4A_4\\
			& + 2\epsilon^2\rho^{-4}\cos2s\cdot(\epsilon f_sA_3 + \epsilon g_sA_4)\\
			&+(2a\epsilon^2\rho^{-6}  + a(\cos2s)^2\epsilon^2\rho^{-2}   - (2a^2 + 4b^2)(\cos2s)^2\epsilon^3\rho^{-5} + 4b^2\epsilon^3\rho^{-9})A_3\\
			& +(2b\epsilon^2\rho^{-6} + b(\cos2s)^2\epsilon^2\rho^{-2}  + 2ab(\cos2s)^2\epsilon^3\rho^{-5} - 4ab\epsilon^3\rho^{-9})A_4 + O(\epsilon^{4}).\\
		\end{aligned}
	\end{equation*}
	\end{lemma}

	\begin{lemma}\label{lemma4.3}
	We have:
		\begin{equation*}
		\begin{aligned}
			<A, d\psi> &=  -2b\cos2s\cdot\epsilon^3\rho^{-2}f_\theta A_3\partial_3\psi - b\cos2s\cdot\epsilon^3\rho^{-2}g_\theta (A_3\partial_4\psi + A_4\partial_3\psi) \\
			&+a\cos2s\cdot\epsilon^{3}\rho^{-2}f_\theta (A_4\partial_3\psi + A_3\partial_4\psi) + 2a\cos2s\cdot\epsilon^{3}\rho^{-2}g_\theta A_4\partial_4\psi\\
			&+(1+b^2(\cos2s)^2\cdot\epsilon^2\rho^{-2} -2ab^2(\cos2s)^2\cdot\epsilon^3\rho^{-5})A_3\cdot\partial_3\psi \\
			&+ (-ab(\cos2s)^2\cdot\epsilon^2\rho^{-2} +2a^2b(\cos2s)^2\cdot\epsilon^3\rho^{-5})(A_3\partial_4\psi + A_4\partial_3\psi) +\\
			& +(1+a^2(\cos2s)^2\cdot\epsilon^2\rho^{-2} - 2a^3(\cos2s)^2\cdot\epsilon^3\rho^{-5})A_4\partial_4\psi\\
			&  + O(\epsilon^{4}).\\
		\end{aligned}
	\end{equation*}
	\end{lemma}

	\begin{lemma}\label{lemma4.4}
		If we denote $d^*dA = (d^*dA)_\tau dy^\tau$, then we have:
		\begin{equation*}
		\begin{aligned}			
			(d^*dA)_1 &= -\partial_3A_{31} - \partial_4A_{41} + O(\epsilon^{2})	\\
			&= \cos\tilde\phi \cdot\epsilon g_s(\frac{V''}{\tilde r} - \frac{V'}{\tilde r^2}) - \sin\tilde\phi\cdot\epsilon f_s(\frac{V''}{\tilde r} - \frac{V'}{\tilde r^2})\\
			&+O(\epsilon^2),\\
			(d^*dA)_2 &= -\partial_3A_{32} - \partial_4A_{42} + O(\epsilon^{2})	\\
			& = \cos\tilde\phi \cdot\epsilon g_\theta(\frac{V''}{r} - \frac{V'}{r^2}) - \sin\tilde\phi\cdot\epsilon f_\theta(\frac{V''}{r} - \frac{V'}{r^2})\\
			&+O(\epsilon^2),\\
		\end{aligned}
	\end{equation*}
	\begin{equation*}
		\begin{aligned}
			(d^*dA)_3 &= -\epsilon^2\rho^{-6}\partial_1A_{13}
			-\epsilon^2\rho^{-2}\partial_2A_{23} - b\cos2s\cdot\epsilon^2\rho^{-2} \partial_3A_{23}\\
			&+a\cos2s\cdot\epsilon^2\rho^{-2}\partial_4A_{23}   +a\cos2s\cdot\epsilon^2\rho^{-2}\partial_2A_{43} \\
			&-(-ab(\cos2s)^2\epsilon^2\rho^{-2}  + 2a^2b(\cos2s)^2\epsilon^3\rho^{-5} )\partial_3A_{43}\\
			&-(1+a^2(\cos2s)^2\epsilon^2\rho^{-2} - 2a^3(\cos2s)^2\epsilon^3\rho^{-5} )\partial_4A_{43} \\
			&-2\epsilon^2\rho^{-4}\cos2s \cdot A_{13} + (\epsilon^4) \\
			&+(2b\epsilon^2\rho^{-6} + b(\cos2s)^2\epsilon^2\rho^{-2}  + 2ab(\cos2s)^2\epsilon^3\rho^{-5} - 4ab\epsilon^3\rho^{-9})A_{43} + O(\epsilon^4),\\
		\end{aligned}
	\end{equation*}
	
	\begin{equation*}
		\begin{aligned}
			(d^*dA)_4 &= -\epsilon^2\rho^{-6}\cdot\partial_1A_{14}
			 -\epsilon^2\rho^{-2}\partial_2A_{24} - b\cos2s\cdot\epsilon^2\rho^{-2}\partial_3A_{24}\\
			&+a\cos2s\cdot\epsilon^2\rho^{-2}\partial_4A_{24}    -b\cos2s\cdot\epsilon^2\rho^{-2}\partial_2A_{34}\\
			&-(1+b^2(\cos2s)^2\cdot\epsilon^2\rho^{-2} -2ab^2(\cos2s)^2\cdot\epsilon^3\rho^{-5})\partial_3A_{34}\\
			&-(-ab(\cos2s)^2\cdot\epsilon^2\rho^{-2} +2a^2b(\cos2s)^2\cdot\epsilon^3\rho^{-5})\partial_4A_{34}\\
			&+ (-2\epsilon^2\rho^{-4}\cos2s )A_{14} +O(\epsilon^4)   \\
			&+(2a\epsilon^2\rho^{-6}  + a(\cos2s)^2\epsilon^2\rho^{-2}   - (2a^2 + 4b^2)(\cos2s)^2\epsilon^3\rho^{-5} + 4b^2\epsilon^3\rho^{-9})A_{34} + O(\epsilon^4).\\
		\end{aligned}
	\end{equation*}
	\end{lemma}

	\begin{lemma}\label{lemma4.5}
		After calculation, we have:
		\begin{equation*}
		\begin{aligned}
			Im(\nabla_A\psi\cdot\bar{\psi})
			&= (\epsilon f_s \cdot \frac{U^2 - VU^2}{\tilde r}\sin\tilde\phi - \epsilon g_s\cdot\frac{U^2 - VU^2}{\tilde r}\cos\tilde\phi)dy^1\\
			&+ (\epsilon f_\theta \cdot \frac{U^2 - VU^2}{\tilde r}\sin\tilde\phi - \epsilon g_\theta\cdot\frac{U^2- VU^2}{\tilde r}\cos\tilde\phi)dy^2\\
			&+(-\frac{U^2}{\tilde r}\sin\tilde\phi - A_3|\psi|^2)dy^3
			+(\frac{U^2}{\tilde r}\cos\tilde\phi - A_4|\psi|^2)dy^4.\\
		\end{aligned}
	\end{equation*}	
	\end{lemma}

	We use the kernels $(\partial_k\psi - iA_k\psi, \partial_k A- \nabla A_k), k =3,4$ ,rather than $(\partial_k\psi , \partial_k A)$ because the inner products of them with the equations are invariant after gauge transformation.
	
	Now we can prove Theorem \ref{Jacobi}:
	\begin{proof}
	
	By Lemma \ref{lemma4.1}, we get:
	\begin{equation*}
		\begin{aligned}
		&<-\Delta \psi, \partial_3\psi - iA_3\psi> \\
		& =  [\epsilon^3\rho^{-6}f_{ss} + \epsilon^3\rho^{-2}f_{\theta\theta} + 2\epsilon^3\rho^{-4}\cos2s\cdot f_s - 2\epsilon^3\rho^{-2}\cos2s\cdot g_\theta \\
		& + 2\epsilon^3\rho^{-6}f - \epsilon^3 \rho^{-2}(\cos2s)^2f]\cdot [\int_{\tilde r} \pi \tilde r(U')^2d\tilde r + \pi\frac{V'}{\tilde r}|_{\tilde r=0})  \\
		& - (2\epsilon^3\rho^{-2}\cos2s\cdot g_\theta + 2\epsilon^3\rho^{-2}(\cos2s)^2\cdot f)\cdot\int_{\tilde r}(-2\pi UU' + \pi VUU')d\tilde r + O(\epsilon^4).
		\end{aligned}
	\end{equation*}

On the other hand, from Lemma \ref{lemma4.2}, we deduce
	\begin{equation*}
		\begin{aligned}
			&< -id^*A\psi, \partial_3\psi - iA_3\psi> = \int_{r,\phi} d^*A(\frac{U^2}{\tilde r} - \frac{VU^2}{\tilde r})\sin\tilde\phi\\
			& = \int_{r,\phi} d^*A\cdot \frac{\sin\tilde\phi}{\tilde r}(-V'' + \frac{V'}{\tilde r}) \\
			& =   \int_{\tilde r} \frac{\pi(V')^2}{\tilde r}\cdot(-\epsilon^3\rho^{-6}f_{ss} - \epsilon^3\rho^{-2}f_{\theta\theta} - 2\epsilon^3\rho^{-4}\cos2s\cdot f_s - 2\epsilon^3\rho^{-6}f + \epsilon^3\rho^{-2}\cos2s\cdot g_\theta )\\
			& + \int_{\tilde r} 4\pi \tilde rV(-V'' + \frac{V'}{\tilde r})(\cos2s)^2\epsilon^3\rho^{-5} - \int_{\tilde r}4\pi \tilde rV \epsilon^3\rho^{-9}(-V'' + \frac{V'}{\tilde r})+ O(\epsilon^4).
		\end{aligned}
	\end{equation*}
	
	By lemma \ref{lemma4.3}, we have:
	\begin{equation*}
		\begin{aligned}
			&<2i(A,d\psi), \partial_3\psi - iA_3\psi> \\
			&=\int_{\tilde r}\cos2s\cdot\epsilon^3\rho^{-2}g_\theta\cdot(-6\pi VUU' + 2\pi V^2UU')\\
			&+\int_{\tilde r} (\cos2s)^2\epsilon^3\rho^{-2}f\cdot(-6\pi VUU' +2\pi V^2UU')\\
			&+\int_{\tilde r} (\cos2s)^2\epsilon^3\rho^{-5}\cdot 4\pi \tilde r^2VUU'+ O(\epsilon^4).
		\end{aligned}
	\end{equation*}

	By direct calculation and using the metric matrix of Fermi coordinate, we get:
	\begin{equation*}
		\begin{aligned}
			&<|A|^2\psi, \partial_3\psi - iA_3\psi> = \int_{\tilde r,\tilde\phi} |A|^2UU'\cos\tilde\phi\\
			&= \int_{\tilde r,\tilde\phi} [ 2A_2A_3\cdot b\cos2s\cdot\epsilon^2\rho^{-2}  - 2A_2A_4\cdot a\cos2s\cdot\epsilon^2\rho^{-2}   \\
			&+A_3^2(1+b^2(\cos2s)^2\cdot\epsilon^2\rho^{-2} -2ab^2(\cos2s)^2\cdot\epsilon^3\rho^{-5})\\
			&+ 2A_3A_4(-ab(\cos2s)^2\cdot\epsilon^2\rho^{-2} +2a^2b(\cos2s)^2\cdot\epsilon^3\rho^{-5})\\
			&+A_4^2(1+a^2(\cos2s)^2\cdot\epsilon^2\rho^{-2} - 2a^3(\cos2s)^2\cdot\epsilon^3\rho^{-5})  ]\cdot UU'\cos\tilde\phi\\
			&=\int_{\tilde r} 2\pi V^2UU'\cdot(\cos2s)^2\cdot\epsilon^3\rho^{-2}f - \int_{\tilde r} 2\pi \tilde r^2V^2UU'(\cos2s)^2\epsilon^3\rho^{-5}\\
			& + \int_{\tilde r}2\pi V^2UU'\cdot\cos2s\cdot\epsilon^3\rho^{-2}g_\theta+ O(\epsilon^4).
		\end{aligned}
	\end{equation*}

	From Lemma \ref{lemma4.4} and Lemma \ref{lemma4.5}, we can get that:
	\begin{equation*}
		\begin{aligned}
			&<d^*dA - Im(\nabla_A\psi\cdot\bar{\psi}),\partial_3 A - \nabla A_3 > \\
			&= \int  (\partial_3A_4 - \partial_4A_3) \cdot [-\epsilon^2\rho^{-6}\cdot\partial_1A_{14}
			 -\epsilon^2\rho^{-2}\partial_2A_{24} - b\cos2s\cdot\epsilon^2\rho^{-2}\partial_3A_{24}\\
			&+a\cos2s\cdot\epsilon^2\rho^{-2}\partial_4A_{24}    -b\cos2s\cdot\epsilon^2\rho^{-2}\partial_2A_{34}\\
			&-(1+b^2(\cos2s)^2\cdot\epsilon^2\rho^{-2} -2ab^2(\cos2s)^2\cdot\epsilon^3\rho^{-5})\partial_3A_{34}\\
			&-(-ab(\cos2s)^2\cdot\epsilon^2\rho^{-2} +2a^2b(\cos2s)^2\cdot\epsilon^3\rho^{-5})\partial_4A_{34}\\
			&+ (-2\epsilon^2\rho^{-4}\cos2s )A_{14}    \\
			&+(2a\epsilon^2\rho^{-6}  + a(\cos2s)^2\epsilon^2\rho^{-2}   - (2a^2 + 4b^2)(\cos2s)^2\epsilon^3\rho^{-5} + 4b^2\epsilon^3\rho^{-9})A_{34} ]\\
			&= \int  (\frac{V'}{\tilde r} ) \cdot [ \epsilon^3\rho^{-6}f_{ss}\cdot\frac{V'}{\tilde r}
			 +\epsilon^3\rho^{-2}f_{\theta\theta}\cdot\frac{V'}{\tilde r}    +b\cos2s\cdot\epsilon^3\rho^{-2}g_\theta\sin\phi\cdot(\frac{V'}{\tilde r})' \\
			&-(1+b^2(\cos2s)^2\cdot\epsilon^2\rho^{-2} -2ab^2(\cos2s)^2\cdot\epsilon^3\rho^{-5})\partial_3A_{34}\\
			&-(-ab(\cos2s)^2\cdot\epsilon^2\rho^{-2} +2a^2b(\cos2s)^2\cdot\epsilon^3\rho^{-5})\partial_4A_{34}\\
			&+ (2\epsilon^3\rho^{-4}\cos2s )f_s\cdot\frac{V'}{\tilde r}   \\
			&+(2a\epsilon^2\rho^{-6}  + a(\cos2s)^2\epsilon^2\rho^{-2}   - (2a^2 + 4b^2)(\cos2s)^2\epsilon^3\rho^{-5} + 4b^2\epsilon^3\rho^{-9})A_{34}]\\
			&= \int_{\tilde r} \pi\epsilon^3 \frac{(V')^2}{\tilde r}\cdot(2\rho^{-6} f_{ss} + 2\rho^{-2} f_{\theta\theta} + 4\rho^{-4}\cos2s\cdot f_s \\
			& -\rho^{-2}\cos2s\cdot g_\theta + 4\rho^{-6}f +\rho^{-2}(\cos2s)^2f  - 6(\cos2s)^2\rho^{-5}\tilde r^2 +4\rho^{-9}\tilde r^2 )+ O(\epsilon^4).
		\end{aligned}
	\end{equation*}

	Finally, from the above results, we deduce:
	\begin{equation*}
		\begin{aligned}
			&<-\Delta_A\psi + \frac{\lambda}{2}(|\psi|^2-1)\psi, \partial_3\psi-iA_3\psi> - <d^*dA - Im(\nabla_A\psi\cdot\bar{\psi}),\partial_3 A - \nabla A_3 >\\
			&= [\epsilon^3\rho^{-6}f_{ss} + \epsilon^3\rho^{-2}f_{\theta\theta} + 2\epsilon^3\rho^{-4}\cos2s\cdot f_s - 2\epsilon^3\rho^{-2}\cos2s\cdot g_\theta \\
		& + 2\epsilon^3\rho^{-6}f - \epsilon^3 \rho^{-2}(\cos2s)^2f]\cdot [\int\pi \tilde r(U')^2d\tilde r + \pi\frac{V'}{\tilde r}|_{\tilde r=0} -3 \pi\frac{(V')^2}{\tilde r} ]  \\
		&+ \int_{\tilde r} \pi\epsilon^3 \frac{(V')^2}{\tilde r}\cdot( 6(\cos2s)^2\rho^{-5}\tilde r^2 - 4\rho^{-9}\tilde r^2 )\\
			& + \int_{\tilde r} 4\pi \tilde rV(-V'' + \frac{V'}{\tilde r})(\cos2s)^2\epsilon^3\rho^{-5}- \int_{\tilde r}4\pi \tilde rV \epsilon^3\rho^{-9}(-V'' + \frac{V'}{\tilde r})\\
			&+\int_{\tilde r} (\cos2s)^2\epsilon^3\rho^{-5}\cdot 4\pi \tilde r^2VUU'- \int_{\tilde r} (\cos2s)^2\epsilon^3\rho^{-5}\cdot 2\pi \tilde r^2V^2UU'+ O(\epsilon^4).
		\end{aligned}
	\end{equation*}
	
	\end{proof}
	
	Noting the exponential decay of $U$ and $V$, we can define the cut-off function $\chi$:
	\begin{equation*}
		\begin{cases}
			\chi = 1, \text{ if } a^2 + b^2 \le \frac{1}{4\epsilon^2\sin2s},\\
			\chi = 0, \text{ if } a^2 + b^2 \ge \frac{1}{2\epsilon^2\sin2s},
		\end{cases}
	\end{equation*}
	and get the conclusion:
	
	\begin{corollary}
		When $\epsilon$ is small enough, the equations
		\begin{equation*}
			\begin{cases}
				<F(u), \chi(\partial_3\psi-iA_3\psi, -\partial_3 A + \nabla A_3)> = 0,\\
				<F(u), \chi(\partial_4\psi-iA_4\psi, -\partial_4 A + \nabla A_4)> = 0,\\
			\end{cases}
		\end{equation*}
	has the form
		\begin{equation*}
			L_H(f, g)  = (M_1, M_2).
		\end{equation*}
	$	M_1,M_2=O(\epsilon)$ is of higher order in $f,g$.
			\end{corollary}

We remark that the principal curvatures of $\Gamma_\epsilon$ in $\pm\epsilon(\sin2s)^\frac32$.
Indeed,
\begin{equation*}
	\begin{aligned}
		&\Gamma_\epsilon = \left(\frac{\cos s}{\epsilon\sqrt{\sin2s}}\Theta, \frac{\sin s}{\epsilon\sqrt{\sin2s}}\Theta\right),\\
		&e_1 = e^{-is}\Theta = -\epsilon(\sin2s)^{\frac32}\partial_s, \text{ }
		e_2 = e^{is}\Theta^\perp = \epsilon\sqrt{\sin2s}\cdot\partial_\theta,\\
		&m = ie^{-is}, n = ie^{is}\Theta^\perp, \text{ }\nabla_{e_1}m = -\epsilon(\sin2s)^\frac32e_1,\\
		&\nabla_{e_2}m = \epsilon\sqrt{\sin2s}\left(\sin2s\cdot e_2 + \cos2s\cdot n\right),\\
		&\nabla_{e_1}n = \epsilon(\sin2s)^\frac32e_2,\\
		&\nabla_{e_2}n = -\epsilon\sqrt{\sin2s}\left(-\sin2s\cdot e_1 + \cos2s\cdot m\right).
	\end{aligned}
\end{equation*}
So we observe that
\begin{equation*}
	\begin{aligned}
		\nabla_{e_1 + e_2}^\tau n = \epsilon(\sin2s)^\frac32(e_1 + e_2),\\
		\nabla_{e_1 - e_2}^\tau n = -\epsilon(\sin2s)^\frac32(e_1 - e_2).
	\end{aligned}
\end{equation*}
For any normal vector $\nu = am + bn$, $a^2 + b^2 = 1$, we have
\begin{equation*}
	\begin{aligned}
		\nabla_{e_1}^\tau\nu = \epsilon(\sin2s)^\frac32(ae_1 + be_2),\\
		\nabla_{e_2}^\tau\nu = \epsilon(\sin2s)^\frac32(be_1 - ae_2).
	\end{aligned}
\end{equation*}
Because $$\det
\begin{pmatrix}
	a & b\\
	b & -a\\
\end{pmatrix}
= -a^2 - b^2 = -1.$$
and mean curvature is 0, we find that for any normal direction $\nu$, the principal curvature is always $\pm\epsilon(\sin2s)^\frac32$.

\section{The approximation solution and errors}
	In this section, we will give the approximation solution and calculate the errors of it. Recall in section 1, we have alreay defined the solution when it is near $\Gamma_\epsilon$: $\psi = \psi^{(1)}(a - \epsilon f, b- \epsilon g)$, $A = A^{(1)}(a - \epsilon f, b- \epsilon g)$. We hope in the whole space, the approximation solution has the form $u = (\psi, A) = (We^{i\varphi}, Z\nabla \varphi)$.
	
	Firstly, we will define $W$ and $Z$:
\begin{equation*}
	\begin{aligned}
		W = \begin{cases}
			U(a - \epsilon f, b - \epsilon g),\quad  r \le \frac1{4\epsilon\sqrt{\sin2s}},\\
			1,\quad r \ge \frac1{2\epsilon\sqrt{\sin2s}},
		\end{cases}
	\end{aligned}
\end{equation*}

\begin{equation*}
	\begin{aligned}
		Z = \begin{cases}
			V(a - \epsilon f, b- \epsilon g),\quad  r \le \frac1{4\epsilon\sqrt{\sin2s}},\\
			1,\quad r \ge \frac1{2\epsilon\sqrt{\sin2s}}.
		\end{cases}
	\end{aligned}
\end{equation*}

In fact, let $\zeta_1$, $\zeta_2$ be a smooth partition of unity on $\mathbb R^4$, subordinate to the open sets $\{r < \frac{1}{2\epsilon\sqrt{\sin2s}}\}$, $\{r > \frac{1}{4\epsilon\sqrt{\sin2s}}\}$. Let $W = \zeta_1U + \zeta_2$, $Z = \zeta_1V + \zeta_2 $ then we have
  \begin{equation*}
  	\begin{aligned}
  		W' = O(e^{-m_\lambda r}),\ Z' = O(e^{-r}),\ \text{ when } r > 2.
  	\end{aligned}
  \end{equation*}

Following \cite{Lin}, we will "extend" the angle function $\phi$ to $\mathbb R^4$ in the following way.
	 We first consider the following equation:
\begin{equation}
	\Delta \omega = 0, \quad in \ \mathbb R^4\backslash\Gamma.
\end{equation}
$\omega$ is a 1-form and satisfies for any oriented closed curve $C$, $\int_C\omega= 2\pi$(the winding number of $C$ around $\Gamma$).
Then we follow the thoughts of Lin F.H. and  Rivi\`ere. T in \cite{Lin} .
Denote:

$G(x) = \frac{1}{8\omega_4\vert x\vert^{2}} $, the fundamental solution of the laplacian operator in $\mathbb R^4$.

$\omega_4$ = the volume of $B_1 \subset \mathbb R^4$.

$\Omega = \mathbb R^4\backslash\Gamma$.

For a locally integrable k-form, we have a $(4 - k)$-current $T_\omega$ given by $T_\omega(\tau) = \int_{\Omega}\omega\wedge\tau$, for any  $\tau \in D^{4-k}(\Omega)$. For any $k$-current $T$ on $\Omega$, $\sigma T$ is a $(k+1)$-current defined by $(\sigma T)(\tau) = T(\delta\tau)$, for any $\tau \in D^{k+1}(\Omega)$. And for any $k$-current $T$ on $\Omega$, $\partial T$ is a $(k-1)$-current defined by $(\partial T)(\tau) = T(d\tau)$, for any $\tau \in D^{k-1}(\Omega)$.

 We can get such a 1-form $\omega$:
\begin{equation}
	\begin{cases}
		\sigma(T_\omega) = 0,\\
		\partial (T_\omega) = 2\pi \Gamma,\\
	\end{cases}
\end{equation}
and for any oriented closed curve $C$, $\int_C\omega = 2\pi$(the winding number of $C$ around $\Gamma$).

In fact, if we suppose for $x\in \Gamma$, $e_1(x), ..., e_4(x)$ is a positive orthonormal base of $\mathbb R^4$, s.t. $e_1(x), e_2(x)$ is a positive base for $T_x\Gamma$. Let $e^1(x),...,e^4(x)$ be the dual base, then we call $\chi(x) = e^1(x)\wedge e^2(x)$ the orientation form of $\Gamma$. Let
\begin{equation*}
	\omega(x) = -2\pi\int_\Gamma *\left[d_x\left(G(x-y)\right)\wedge\chi(y)\right]dH^2(y),
\end{equation*}
for $x\in \mathbb R^4$.

>From (5), we know
\begin{equation}
	\begin{cases}
		\delta \omega = 0,\\
		\delta d\omega = 0.\\
	\end{cases}
\end{equation}
So $\Delta \omega = 0$.

Then, we have $\int_C\omega - d\phi = 0$ , for any oriented closed curve $C\subset\{\rho_1 = \rho_2\}$, so it is equal to $dh$ for some smooth function $h$. It follows from the growth of $\omega$ and $d\phi$ near $\Gamma$ that $df$ is $C^{0,1}$ on $\{\rho_1 = \rho_2\}$. Observe that $\omega - d\phi$ satisfies:
\begin{equation}
	\begin{cases}
		\Delta(\omega - d\phi) = -d(\Delta\phi), \ $in$ \ \{\rho_1 \le \rho_2\},\\
		(\omega - d\phi)_\top = dh, \ $on$ \ \{\rho_1 = \rho_2\},\\
		\left(\delta\left(\omega - d\phi\right)\right)_\top = \Delta \phi,	\ $on$ \ \{\rho_1 = \rho_2\}.
	\end{cases}
\end{equation}
Let us solve the Dirichlet problem
\begin{equation*}
	\begin{cases}
		\Delta \tilde\varphi_1 = -\Delta \phi, $ in  $\{\rho_1 < \rho_2\},\\
		\tilde\varphi_1\vert_{\{\rho_1 = \rho_2\}} = h,	$ in $\{\rho_1 \le  \rho_2\}.
	\end{cases}
\end{equation*}
 Then $d\tilde\varphi_1$ satisfies (9), so $\omega = d\phi + d\tilde\varphi_1$. Similarly, we can get $\omega = d\phi + d\tilde\varphi_2$ in $\{\rho_1 \ge  \rho_2\}$, where $\tilde\varphi_2$ satisfies:
 \begin{equation*}
	\begin{cases}
		\Delta \tilde\varphi_2 = -\Delta \phi, $ in  $\{\rho_1 > \rho_2\},\\
		\tilde\varphi_2\vert_{\{\rho_1 = \rho_2\}} = h,	$ in $\{\rho_1 \ge   \rho_2\}.
	\end{cases}
\end{equation*}
Denote $\tilde\varphi = $
\begin{equation*}
	\begin{cases}
		\phi + \tilde\varphi_1,\  $ in$ \ \{\rho_1 \le  \rho_2\},\\
		\phi + \tilde\varphi_2,\  $ in$ \ \{\rho_1 \ge  \rho_2\}.\\
	\end{cases}
\end{equation*}
then $\tilde\varphi$ is a solution of (5), and $d\tilde\varphi = \omega \in C^\infty(\mathbb R^4\backslash \Gamma)$, $\tilde\varphi \in C^\infty_{loc}(\mathbb R^4\backslash\Gamma)$.

We have the following results:
\begin{lemma}\label{the boundary condition}
	 When it is on $\{\rho_1 = \rho_2\}$, we have
	 \begin{equation*}
		\begin{aligned}
			&\omega_\top(x)  = \frac{\pi}{8\omega_4}\int_{\rho \ge 1, 0\le\theta < 2\pi} |x-y|^{-4}\cdot
			 \{[-2\rho^{-3}\rho_1^2 + 2\sqrt{2}\rho^{-2}\rho_1\cos(\theta - \theta_1)  \\
			 &+ 2\rho_1^2\rho - 2\sqrt{2}\rho_1\cos(\theta - \theta_2)  ]d\theta_1(x)\\
			&+ [2\rho_1^2\rho - 2\sqrt{2}\rho_1\cos(\theta - \theta_2) - 2\rho_1^2\rho^{-3} + 2\sqrt{2}\rho_1\rho^{-2}\cos(\theta - \theta_1) ]d\theta_2(x)\}d\rho d\theta.
		\end{aligned}
	\end{equation*}
\end{lemma}
The proof will be given in the Appendix. Next we estimate $\tilde \varphi_1$.
	
\begin{theorem}
	$\vert \tilde\varphi_1\vert < \infty$ in $\mathbb R^4$.
\end{theorem}
\begin{proof}
	Firstly, we will estimate $\Delta\phi$.
	By previous arguments, we know that when $\rho$ is large enough, $\Delta \phi = O(\rho^{-4})$. Besides, $\phi$ is smooth away from $\Gamma$. So we only need to estimate $\Delta \phi$ near $\Gamma$ in $B_R$.
	\begin{equation*}
		\begin{aligned}
			\Delta \phi &= (w_{3,3} + w_{3,4} + w_{4,3} + w_{4,4})\phi\\
			&= 2b(\sin2s)^{\frac52}(\cos2s)^2 + O(r^{-2}).
		\end{aligned}
	\end{equation*}
	Therefore, $\vert\Delta\phi\vert_{C^{0,\alpha}} < \infty$.
	
	Next we will show that $\vert h\vert < \infty$.
	
	By Lemma \ref{the boundary condition} , we can get that in $\{\rho_1 = \rho_2\}$, $\omega(x) \in $ span$\{d\rho_1 - d\rho_2, d\theta_1 + d\theta_2\}$.
	Because of $(\omega - d\phi)_\top = dh$ on $\{\rho_1 = \rho_2\}$, we know:
	\begin{equation*}
		\begin{aligned}
			\vert h_{\theta_1}\vert \le C, \text{ }\vert h_{\theta_2}\vert \le C, \text{ }\vert h_{\rho_1}\vert = 0.\\
		\end{aligned}
	\end{equation*}
	If we fix $h = 0$ at (1, 0, 1, 0), then $h < \infty$ on $\{\rho_1 = \rho_2\}$.
	So by standard elliptic estimates, we have $\vert \tilde\varphi_1\vert < \infty$.

\end{proof}
Therefore, we know that the same results hold for $\tilde\varphi_2$ and $\tilde\varphi$.

	After rescaling, we can get similar results for $\Gamma_\epsilon$. For simplicity, we still denote it by $\tilde\varphi$.
		
	\begin{equation*}
		\begin{aligned}
			\varphi = \zeta_1\phi(a-\epsilon f, b- \epsilon g) + \zeta_2\tilde \varphi.
		\end{aligned}
	\end{equation*}
In order to estimate $F(u)$, we define the following weighted norm:
\begin{equation*}
	\begin{aligned}
		\|h\|_{**} &=  \sup_{P}\rho(P)^2e^{\delta r}\|h\|_{L^q(B_1(P))},
	\end{aligned}
\end{equation*}
where $ 1 < q < 2$.

For 1-form $B$, we define:
\begin{equation*}
	\begin{aligned}
		\|(B, h)\|_{**} = \|B\|_{**} + \|h\|_{**}.
	\end{aligned}
\end{equation*}
In order to control $f, g$, we define the following norm $W^{j, q}_k$:
	\begin{equation*}
		\begin{aligned}
			&\|f\|_{0, k, q} := \sup_{P \in \Gamma_\epsilon}\rho(P)^k\|f\|_{L^q(\rho(P) \le  \rho \le \rho(P) + 1 )},\\
		&\|f\|_{2, k, q} := \sup_{P \in \Gamma_\epsilon}\rho(P)^k\|f\|_{L^q(\rho(P) \le  \rho \le \rho(P) + 1 )} + \sup_{P \in \Gamma_\epsilon}\rho(P)^{k + 1}\|\nabla f\|_{L^q(\rho(P) \le  \rho \le \rho(P) + 1 )}\\
		&+ \sup_{P \in \Gamma_\epsilon}\rho(P)^{k + 2}\|\nabla^2 f\|_{L^q(\rho(P) \le  \rho \le \rho(P) + 1 )},\\
		&\|(f, g)\|_{2, k, q} := \|f\|_{2, k, q} + \|g\|_{2, k, q}.
		\end{aligned}
	\end{equation*}
And we always assume that
	\begin{equation}\label{disturbance}
		\|(f, g)\|_{2, 1, q} \le C.
	\end{equation}

	Now, we have the following results for the error:

	\begin{theorem}
		\begin{equation}\label{error}
			\|F(u)\|_{**} < C\epsilon^2.
		\end{equation}
	\end{theorem}
	\begin{proof}

		\begin{itemize}
		\item 1.When $r > \frac{1}{2\epsilon\sqrt{\sin2s}} $, we can get
		\begin{equation*}
			F(u) = 0.
		\end{equation*}
		
	\item 2.When $\frac{1}{4\epsilon\sqrt{\sin2s}}  < r < \frac{1}{2\epsilon\sqrt{\sin2s}} $,
		\begin{equation*}
			\begin{aligned}
				1 - Z, 1 - W = O(e^{-\delta r}) , \text{ where } 0 < \delta < \min\{m_{\lambda}, 1\}.
			\end{aligned}
		\end{equation*}
  Thus we can get that
	\begin{equation*}
		F(u) = O(e^{-\delta r}).
	\end{equation*}
	\item 3.When $2< r < \frac{1}{4\epsilon\sqrt{\sin2s}}$, then
	\begin{equation*}
		W = U, Z = V, \varphi = \phi.
	\end{equation*}
	Similarly, we know at least $F(u) = O(e^{-\delta r})$. So if we want to get more accurate estimates, we only need to consider the exponential decay items in $F(u)$, which we have already calculated in section 1.
	Then we get the errors:
\begin{equation*}
	\begin{aligned}
		F(u) = O(\epsilon^2\rho^{-2}e^{-\delta r}).\\
	\end{aligned}
\end{equation*}
	\item 4.When $r < 2$, observe that the approximate solution may not be $C^{\infty}$ smooth near $\Gamma_\epsilon$. For example, when it is close to $\Gamma_\epsilon$, we have:
	\begin{equation*}
		\Delta \psi = \Delta (We^{i\varphi}) = \Delta W\cdot e^{i\varphi} - We^{i\varphi}\cdot|\nabla\varphi|^2 + 2ie^{i\varphi}\nabla W\cdot\nabla \varphi + ie^{i\varphi}W\Delta \varphi.
	\end{equation*}
	
	Then the item $ We^{i\varphi}\cdot|\nabla\varphi|^2 = O(\tilde r^{-1})$.
	This is different from the case in Allen-Cahn equation. As a result, we use the $L^p$ norm rather than H\"older norm.
	
	We only need to care about the $O(\tilde r^{-1})$ items, and get:
	\begin{equation*}
		\begin{aligned}
			-\Delta_A\psi + \frac{\lambda}{2}(|\psi|^2-1)\psi = O(\epsilon^2\rho^{-2}\tilde r^{-1}),\\
			d^*dA - Im(\nabla_A\psi\cdot\bar\psi) =  O(\epsilon^2\rho^{-2}).
		\end{aligned}
	\end{equation*}

	\end{itemize}

Therefore, we get that:
\begin{equation*}
	\|F(u)\|_{**} < C\epsilon^2.
\end{equation*}

	\end{proof}

Remark: We distinguish $r$ and $\tilde r$ only when it is close to $\Gamma_\epsilon$. Otherwise, we see them as the same.

\section{The linearized operator}

>From now on, we will apply standard infinite dimensional Lyapunov-Schmidt reduction. See \cite{Liu}, \cite{Del Pino}.

Denote $L$ as the linearization of the Ginzburg-Landau equations at an approximate solution $ u = (A, \psi)$. Then

\begin{equation*}
	\begin{aligned}
		L(B, \eta) =
		\begin{pmatrix}
			d^*dB - Im(\nabla\psi\cdot\bar \eta + \nabla \eta\cdot\bar\psi) + A(\bar \psi \eta + \psi\bar \eta) + B|\psi|^2\\
			-\Delta_A\eta - id^*B\cdot\psi + 2i<B, d\psi> + 2<A, B>\psi + \frac{\lambda}{2}(\psi^2\bar \eta  + 2|\psi|^2\eta - \eta)
		\end{pmatrix}.
	\end{aligned}
\end{equation*}

Define $v = (\tilde A, \tilde \psi) := u + w = (A, \psi) + (B, \eta)$.  Note that the first equation is not elliptic, so we always need to add some items so that the new operator is uniformly elliptic.

Denote the new operator $T : \Omega^0(\mathbb R^4, \mathbb R) \rightarrow \Omega^1(\mathbb R^4, \mathbb R) \oplus \Omega^0(\mathbb R^4, \mathbb C) $:
\begin{equation*}
	T(v) = (dv, iv\tilde\psi).
\end{equation*}
We remark that in fact that the operator is the kernel of $L$ induced by the gauge transformation:
\begin{equation*}
	\psi \rightarrow e^{iv}\psi,
\end{equation*}
\begin{equation*}
	A\rightarrow A + \nabla v.
\end{equation*}

The adjoint of $T$ is given by:
\begin{equation*}
	T^* (B, \eta) = d^*B - Im(\bar \eta\tilde\psi).
\end{equation*}
And \begin{equation*}
	TT^*(B, \eta) = \begin{pmatrix}
		dd^*B - d[Im(\bar\eta\psi)]\\
		i\tilde\psi\cdot(d^*B - Im(\bar\eta\psi)
	\end{pmatrix}.
\end{equation*}
Finally we define $\mathbb L(B, \eta) := L(B, \eta) + TT^*(B, \eta)$:
\begin{equation*}
	\mathbb L(B, \eta) = \begin{pmatrix}
		-\Delta B + 2Im(\overline{\nabla_A\psi}\cdot \eta) + B|\psi|^2\\
		-\Delta_A\eta + 2i<B,\nabla_A\psi>  + \frac{1}{2}(\lambda - 1)\psi^2\bar \eta + (\lambda + \frac{1}{2} )|\psi|^2\eta - \frac{\lambda}{2}\eta
	\end{pmatrix}.
\end{equation*}

Here we remark that, we only use the linear part of $TT^*(B, \eta)$ and the nonlinear part will be added to the following $N(w)$ part.

So by \eqref{error}, we can define the following weighted norm:

\begin{equation*}
	\begin{aligned}
		\|\eta\|_{*} &=   \sup_{P}\rho(P)^2e^{\delta r}\|\eta\|_{W^{2,q}(B_1(P))},
	\end{aligned}
\end{equation*}
where $1 < q < 2$.

Similarly, we define:
\begin{equation*}
	\|(B, \eta)\|_{*} = \|B\|_{*} + \|\eta\|_{*}.
\end{equation*}

We hope the "true" solution of the Ginzburg-Landau equation is $u + w :=(A, \psi) + (B, \eta)$, then we get the new equations for $w = (B, \eta)$:
\begin{equation*}
	\begin{aligned}
		F(u + w) + TT^*(B, \eta) = F(u) + \mathbb Lw + N(w) = 0,
	\end{aligned}
\end{equation*}

where
\begin{equation}\label{nonlinear items}
	N(w) = \begin{pmatrix}
		-Im(\bar\eta\nabla_A\eta) + B(2Re(\bar\psi\eta) + |\eta|^2) \\
		\lambda(2\psi\bar\eta + \bar\psi\eta + |\eta|^2)\eta + |B|^2(\psi + \eta) + i[-\eta Im(\bar\eta\psi) + 2B\cdot\nabla_A\eta]\\
	\end{pmatrix}.
\end{equation}

We now recall the nondegeneracy property of the linearized operator.

\begin{theorem}[See \cite{G}]
	For all $\lambda > 0$, the $\pm1$-vortex is stable.
\end{theorem}

Naturally, we then consider the linearized operator in the $\mathbb R^2 \times \mathbb R^2$ case. By Fourier transform, we can get:
\begin{corollary}\label{linearized operator}
	Let $(x, y) \in \mathbb R^2 \times \mathbb R^2$ and denote $(A^{(1)}, \psi^{(1)})$ the $+1$ vortex solution to the Ginzburg-Landau equations.  Define the linear operator
	\begin{equation*}
		\begin{aligned}
			L^{(1)}(B, \eta):= \begin{pmatrix}
				-\Delta B + 2Im(\overline{\nabla_A\psi}\cdot\eta) + B|\psi|^2\\
			-\Delta_A\eta + 2i<B,\nabla_A\psi>  + \frac{1}{2}(\lambda - 1)\psi^2\bar \eta + (\lambda + \frac{1}{2} )|\psi|^2\eta - \frac{\lambda}{2}\eta
			\end{pmatrix}.
		\end{aligned}
	\end{equation*}
	where $(A, \psi) = (A^{(1)}(y), \psi^{(1)}(y))$. Then if $\|(B,\eta)\|'_* < \infty$, we have $(B, \eta) = 0$.
\end{corollary}
Here $\|\eta\|'_* :=  \sup_{P} e^{\delta |y|}\|\eta\|_{W^{2,q}(B_1(P))}$. With this  , we can give the most important result in this section.
\begin{theorem}\label{thm est}
	Let $\epsilon > 0$ be small. Suppose $\|(B ,\eta)\|_* < \infty$, $\|(D, h)\|_{**} < \infty$, and

	\begin{numcases}{}
		\mathbb L (B, \eta) = (D, h),\\
		\label{kernel1}\int_{r, \phi  } \zeta_1<B, \partial_3A - \nabla A_3> + \zeta_1Re[\bar \eta(\partial_3\psi - iA_3\psi )] = 0,\\
		\label{kernel2}\int_{r , \phi }\zeta_1<B, \partial_4A - \nabla A_4> + \zeta_1Re[\bar \eta(\partial_4\psi - iA_4\psi )] = 0,
	\end{numcases}
for any $(s, \theta)$.
Then $\|(B, \eta)\|_{**} \le C\|(D, h)\|_{*}$, where $C$ is independent of $\epsilon$, $D$ and $h$.
\end{theorem}
\begin{proof}
	By  $L^p$ theory, it suffices to show that
	\begin{equation*}
		\begin{aligned}
			&\sup_{ P} \rho^2e^{\delta r}( \|B\|_{L^q(B_1(P))} + \|\eta\|_{L^q(B_1(P))}) \\
		 &\le C\sup_{ P} \rho^2e^{\delta r}( \|D\|_{L^q(B_2(P))} + \|h\|_{L^q(B_2(P))}).
		\end{aligned}
	\end{equation*}
	To prove this, we argue by contradiction and and assume that there exists a sequence $\epsilon_j$, $(B^{(j)}, \eta^{(j)})$, $(D^{(j)}, h^{(j)})$, such that:
	\begin{equation*}
		\begin{aligned}
			&\sup_{ P} \rho^2e^{\delta r}( \|B^{(j)}\|_{L^q(B_1(P))} + \|\eta^{(j)}\|_{L^q(B_1(P))}) = 1,
		\end{aligned}
	\end{equation*}
	\begin{equation*}
		\begin{aligned}
			&\sup_{ P} \rho^2e^{\delta r}( \|D^{(j)}\|_{L^q(B_2(P))} + \|h^{(j)}\|_{L^q(B_2(P))}) \rightarrow 0,
		\end{aligned}
	\end{equation*}
	where $\mathbb L(B^{(j)}, \eta^{(j)}) = (D^{(j)}, h^{(j)})$.
	
	By the decay of $(B^{(j)}, \eta^{(j)})$, we know that $(B^{(j)}, \eta^{(j)}) \in W^{2,q}(\mathbb R^4) \subset\subset H^{1}(\mathbb R^4)$. So there exists a subsequence, such that $(B^{(j)}, \eta^{(j)})$ converges to some $(B^{(0)}, \eta^{(0)})$ in $H^1(\mathbb R^4)$. For simplicity, we still denote them as $(B^{(j)}, \eta^{(j)})$.
	
	 So for any $j$, there exists $P_j$, such that:
	 \begin{equation*}
	 	\|B^{(j)}\|_{L^q(B_1(P_j))} + \|\eta^{(j)}\|_{L^q(B_1(P_j))} \ge \frac{1}{2}\rho^{-2}e^{-\delta r}(P_j).
	 \end{equation*}	
	
	As a result, there are three cases for the convergence of $r(P_j)$:
	\begin{itemize}
		\item 1. $r(P_j)$  will converge to some finite constant $r_0 > 0$ , then we can always assume $P_j$ converges to $P_0$.
		At this case, we assume that $x_j \in \Gamma_{\epsilon_j}$ and $P_j = (x_j, y_j), y_j \in \mathbb R^2$ and define:
		\begin{equation*}
			\begin{aligned}
				\tilde B^{(j)}(x, y) = \rho(P_j)^2 B^{(j)}(x+x_j, y),\\
				\tilde \eta^{(j)} (x, y) = \rho(P_j)^2\eta^{(j)}(x + x_j, y),
			\end{aligned}
		\end{equation*}
		then we have
		\begin{equation}\label{case1}
			\|\tilde B^{(0)}\|_{L^q(B_1(0, y_0))} + \|\tilde\eta^{(0)}\|_{L^q(B_1(0, y_0))} \ge \frac{1}{4}e^{-\delta r_0}(0, y_0).
		\end{equation}
		Then let $j \rightarrow \infty$	, we get:
		\begin{equation*}
			\lim_{j\rightarrow \infty} \mathbb L(\tilde B^{(j)}, \tilde\eta^{(j)})  = L^{(1)}(\tilde B^{(0)}, \tilde\eta^{(0)}) = 0, \text{in } \mathbb R^2 \times  \mathbb R^2.
		\end{equation*}
		And because of \eqref{kernel1} and \eqref{kernel2}, as well as Corollary  \ref{linearized operator}, we know that $(\tilde B^{(0)}, \tilde\eta^{(0)}) = 0$, which contradicts to \eqref{case1}.
		
		\item 2. $r(P_j)$  will tend to infinity.
		We assume that $x_j \in \Gamma_{\epsilon_j}$ and $P = (x_j, y_j), y_j \in \mathbb R^2$ and define a sequence of pairs $(\tilde B^{(j)}, \tilde\eta^{(j)})$:
		\begin{equation*}
			\begin{aligned}
				\tilde B^{(j)}(x, y) = \rho(P_j)^2 e^{\delta r(P_j)}\tilde B^{(j)}(x + x_j, y + y_j),\\
				\tilde \eta^{(j)}(x, y) = \rho(P_j)^2 e^{\delta r(P_j)}\tilde \eta^{(j)}(x + x_j, y + y_j).\\
			\end{aligned}
		\end{equation*}
		Then $(\tilde B^{(j)}, \tilde\eta^{(j)})$  will converge to some $(\tilde B^{(0)}, \tilde\eta^{(0)})$, which satisfies:
		
		\begin{equation*}
			\begin{cases}
				-\Delta   B + B = 0, \text{ in  } \mathbb R^2 \times \mathbb R^2, \\
				-\Delta \eta + \frac{1}{2}(\lambda - 1)\bar {\eta} + \frac{1}{2}(\lambda + 1)\eta  = 0, \text{ in  } \mathbb R^2 \times \mathbb R^2,\\
				 \sup_P e^{-\delta |y|}(\| B\|_{L^q(B_3(P))} + \|\eta\|_{L^q(B_3(P))}) \le 1, \text{ in  } \mathbb R^2 \times \mathbb R^2.
			\end{cases}
		\end{equation*}
		Thus $(\tilde B^{(0)}, \tilde\eta^{(0)}) = 0$. But we have:
		\begin{equation*}
			 (\| \tilde B^{(0)}\|_{L^q(B_3(0))} + \|\tilde\eta^{(0)}\|_{L^q(B_3(0))}) >  \frac{1}{4}.\\
		\end{equation*}
		This is a contradiction.
	\end{itemize}

\end{proof}	
		
	Denote $T_1 = \zeta_1(\partial_3A - \nabla A_3, \partial_3\psi - iA_3\psi )$, $T_2 = \zeta_1(\partial_4A - \nabla A_4, \partial_4\psi - iA_4\psi )$.
	Then we have:
	\begin{equation*}
		\begin{aligned}
			\int_{r, \phi }<T_1, T_2> = O(\epsilon^2\rho^{-2}).
		\end{aligned}
	\end{equation*}
	
	Define orthogonal projections:
	\begin{equation*}
		\begin{aligned}
			\pi &:= L^2 -\text{ orthogonal projection onto span}\{T_1, T_2\},\\
		\pi^{\perp} &:= 1 - \pi.
		\end{aligned}
	\end{equation*}
		
	Then we will solve the projected linear problem:
	\begin{theorem}\label{existence theorem}
		 If  $\|(D, h)\|_{**} < \infty$ and $\pi(D, h) = 0$, then for $\epsilon$ sufficiently small there exists a unique solution $(B, \eta)$ to
		 \begin{equation}\label{equation for priori}
		 	\begin{cases}
		 		\pi^{\perp}\mathbb L(B, \eta) = (D, h),\\
		 		\int_{r, \phi} <(B, \eta), T_j> = 0, \text{ for } j = 1, 2,
		 	\end{cases}
		 \end{equation}
		  and we have
		 \begin{equation*}
		 	\|(B, \eta)\|_* \le C \|(D, h)\|_{**},
		 \end{equation*}
		 with $C$ only depending on $q$, $\gamma$, $\delta$.
	\end{theorem}
	\begin{proof}
		
		The proof is standard. We first solve the equation in bounded domains.
		\begin{equation}\label{bounded equ}
			\begin{cases}
				\pi^{\perp}\mathbb L(B, \eta) = (D, h) \text{  in } B_M(0),\\
				B, \eta = 0 \text{  on } \partial B_M(0).
			\end{cases}
		\end{equation}
		And define the space
		\begin{equation*}
			\mathcal H := \{w = (B, \eta): B, \eta \in H^1_0(B_M), \pi^{\perp}(B, \eta) = 0\}.
		\end{equation*}
		
		We equip $\mathcal H$ with the inner product:
		\begin{equation*}
			[w_1, w_2] := \int_{B_M(0)} <\nabla B_1, \nabla B_2> + Re\int_{B_M(0)}\nabla\eta_1\cdot\nabla\eta_2.
		\end{equation*}
		Then $\mathcal H$ is a Hilbert space.
		
		The equation can be rewritten in its variational form. Namely, for any $\tilde w = (\tilde B, \tilde \eta) \in \mathcal H$:
		\begin{equation*}
			\begin{aligned}
				&\int_{B_M(0)} <\nabla B, \nabla \tilde B> + Re\int_{B_M(0)}\overline{\nabla_A\eta}\cdot\nabla_A\tilde\eta \\
				&+ \int_{B_M(0)} <2Im(\overline{\nabla_A\psi}\cdot \eta), \tilde B> + \int_{B_M(0)}<B|\psi|^2, \tilde B> - Re\int_{B_M(0)}2i<B, \overline{\nabla_A\psi}>\tilde\eta \\
				&+ Re\int_{B_M(0)}\frac{1}{2}(\lambda - 1)\bar \psi^2\eta\tilde\eta + Re\int_{B_M(0)}(\lambda + \frac{1}{2})|\psi|^2\bar\eta\tilde\eta - Re\int_{B_M(0)}\frac{\lambda}{2}\bar\eta\tilde\eta \\
				&= \int_{B_M(0)}<D, \tilde B> + Re\int_{B_M(0)}\bar h\tilde\eta.
			\end{aligned}
		\end{equation*}
		
		Then we define the linear form $<k(x)w, \cdot>$ on $\mathcal H$:
		\begin{equation*}
			\begin{aligned}
				<k(x)w, \tilde w> :=& \int_{B_M(0)} <2Im(\overline{\nabla_A\psi}\cdot \eta), \tilde B> + \int_{B_M(0)}<B|\psi|^2, \tilde B> \\
				&- Re\int_{B_M(0)}2i<B, \overline{\nabla_A\psi}>\tilde\eta \\
				&+ Re\int_{B_M(0)}\frac{1}{2}(\lambda - 1)\bar \psi^2\eta\tilde\eta + Re\int_{B_M(0)}(\lambda + \frac{1}{2})|\psi|^2\bar\eta\tilde\eta - Re\int_{B_M(0)}\frac{\lambda}{2}\bar\eta\tilde\eta. \\
			\end{aligned}
		\end{equation*}
		Similarly, we denote $z := (D, h)$ and $<z, \cdot>$ the linear form defined by:
		\begin{equation*}
			\begin{aligned}
				<z, \tilde w> := \int_{B_M(0)}<D, \tilde B> + Re\int_{B_M(0)}\bar h\tilde\eta.
			\end{aligned}			
		\end{equation*}
		Then the equation can be written as
		\begin{equation*}
			[w, \tilde w] + <k(x)w, \tilde w> = <z, \tilde w>, \forall \tilde w\in \mathcal H.
		\end{equation*}
		
		Using the Riesz representation theorem, we can find a bounded linear operator $K$ on $\mathcal H$, and an element $Z$ of $\mathcal H$ depending linearly on $z$. Thus the equation has the following operational form:
		\begin{equation}\label{operator equation}
			w + K(w) = z.
		\end{equation}
		Besides because of the compact Sobolev injection $H^1_0(B_M(0)) \subset\subset L^2(B_M(0))$, we get that $K$ is compact. Then by Fredholm's alternative, we deduce the existence of $w$ satisfying \eqref{operator equation} if the homogeneous equation only has the trivial solution. In this case, we have $z = 0$.
		
		Now let us rewrite the equation \eqref{equation for priori}:
		\begin{equation*}
			\mathbb L(B, \eta) = c_1(s,\theta)T_1 + c_2(s, \theta)T_2.
		\end{equation*}

		We will estimate $c_1$ and $c_2$. Here we always "modify" $B_M(0)$, such that $ \{(s, \theta, r, \phi): r < 3\} \subset B_M(0)$ for all $(s, \theta, 0, 0) \in B_M(0)$. (For example, we can use the Fermi coordinates, and use the region $\{\rho < M\}$).
		 Then multiply the equation with $T_1$ and integral on $ \{(s, \theta, r, \phi): r < 3\}$ for  any $(s, \theta) \in \Gamma_\epsilon \cap B_M(0)$.
		\begin{equation*}
			\begin{aligned}
				c_1 = \frac{1}{c^*_1}\int_{ r, \phi}<\mathbb L(B, \eta), T_1> + O(\epsilon^2\rho^{-2})c_2,\\
				c_2 = \frac{1}{c^*_2}\int_{ r, \phi}<\mathbb L(B, \eta), T_2> + O(\epsilon^2\rho^{-2})c_1,\\
			\end{aligned}	
		\end{equation*}
		where\begin{equation*}
			c^*_j = \int_{ r, \phi} <T_j, T_j>, j = 1,2.
		\end{equation*}
		
		So we only need to estimate $\mathbb LT_1$. Note that there are no first order differentiation items. So we can always divide $\Delta$ and $\Delta_A$ into two parts individually.
		\begin{equation*}
			\begin{aligned}
				\mathbb L := \mathbb L_1 + \mathbb L_2,\\
			\end{aligned}
		\end{equation*}
		where
		\begin{equation*}
			\mathbb L_1(B,\eta) = \begin{pmatrix}
				-\Delta B + \Delta_{\Gamma_\epsilon} B\\
				-\Delta_A\eta + \Delta_{A, \Gamma_\epsilon}\eta
			\end{pmatrix}.
		\end{equation*}
		Then for any fixed $(s, \theta)$
		\begin{equation*}
			\begin{aligned}
				\int_{r, \phi}<\mathbb L_2(B, \eta), T_1> &= \int_{r, \phi}<(B, \eta), \mathbb L_2T_1>\\
				&+ O(\epsilon^2\rho^{-2})\int_{r < \frac{1}{2\epsilon\sqrt{\sin2s}} }|(B, \eta)|^q + |\nabla (B, \eta)|^q.
			\end{aligned}
		\end{equation*}
		By the condition \eqref{equation for priori}, we know
		 \begin{equation*}
		 	\begin{aligned}
		 		\|\int_{r, \phi}<\mathbb L_1(B, \eta), T_1>\|_{0, 2, q} = O(\epsilon^2)\|(B, \eta)\|_*.
		 	\end{aligned}
		 \end{equation*}
		Note that when $\epsilon$ tends to zero, $\mathbb L_2T_1$ also tends to zero. So we only need to consider the $O(\epsilon)$ items in $\mathbb L_2T_1$. Recalling the calculation in Section 1, we finally get:
		\begin{equation*}
			\mathbb L_2T_1 = O(\epsilon^2\rho^{-2}e^{-\delta r}).
	\end{equation*}
		Therefore
		\begin{equation*}
			\|c_j\|_{0,2,q} = O(\epsilon^2)\|(B, \eta)\|_*.
		\end{equation*}
		By Theorem \ref{thm est}, we have:
		\begin{equation*}
			\|(B,\eta)\|_{*} \le C\epsilon^2\|(B,\eta)\|_{*}.
		\end{equation*}
		So we get that $(B, \eta) = 0$ for the homogeneous equation. Then for any large enough $M$, we obtain the existence of a solution of \eqref{bounded equ}  satisfying:
		\begin{equation*}
			\|(B, \eta)\|_{*} \le C\|(D, h)\|_{**}.
		\end{equation*}
		with C independent of $M$. Note that in the previous arguments the norm $\|\cdot\|_*$, $\|\cdot\|_{**}$ are adapted to deal with the bounded domain situation. Similarly in the proof of Theorem \ref{thm est}, we can get a subsequence such that $(B, \eta) \rightharpoonup (B, \eta)$ in $H^1_{loc}(\mathbb R^4)$ with  $(B, \eta)$ solving \eqref{equation for priori} in $\mathbb R^4$.
		
	\end{proof}

	\section{A projected nonlinear problem}
	In this section we will solve the equation:
	\begin{equation}\label{nonlinear equ}
		\pi^{\perp} [F(u) + \mathbb Lw + N(w)] = 0.
	\end{equation}
	\begin{theorem}\label{thm of existence of the nonlinear problem}
		If $\epsilon$ is small enough, for any $f$, $g$ satisfying \eqref{disturbance} , there exists one $w = (B, \eta)$ such that \eqref{nonlinear equ} is true and $\pi(w) = 0$ for any $(s, \theta)$. And we have
		\begin{equation*}
			\|w\|_{*} \le C\epsilon^2.
		\end{equation*}
	\end{theorem}
	\begin{proof}
		
		The proof is standard. Here we always choose $1 < q <2$ is close to 2. For example, we can let $q = \frac{19}{10}$.  Different from the cases in \cite{Liu}, where Holder norm is widely used, we need to check that $\|N(w)\|_{**} \le C\delta_0^2$ if $\|w\|_* \le \delta_0$.
		
		By \eqref{nonlinear items}, we know that we only need to show:
		\begin{equation}\label{the L^q norm of N(w)}
			|\eta|^3, |\nabla\eta\cdot\eta| \le C\delta_0^2.
		\end{equation}
		Then by Sobolev embedding, we know $W^{2,q} \hookrightarrow W^{1, \frac{4q}{4 - q}}\hookrightarrow L^{\frac{4q}{4-2q}}$. Thus by Holder inequality, as long as $\frac{4}{3} \le q < 2$ and $\epsilon$ is small enough, \eqref{the L^q norm of N(w)} holds.

		When $\epsilon$ is small enough, by Theorem \ref{existence theorem}, $\mathbb L $ is invertible. Then we can rewrite the equation \eqref{nonlinear equ} as a fixed point equation:
		\begin{equation*}
			\begin{aligned}
				&\pi^{\perp}\mathbb Lw = - \pi^{\perp}F(u) - \pi^{\perp}N(w)\\
				\Rightarrow & w = S(w),
			\end{aligned}
		\end{equation*}
		where $S(w):= \mathbb L^{-1}[- \pi^{\perp}F(u) - \pi^{\perp}N(w)]$.
		
		Denote the ball $B^{\perp}_{\delta_0} = \{w: \|w\|_{*} < \delta_0\} \cap \{w: \pi(w)= 0\}$, where $\delta_0 > 0$ is  small. Then for $\epsilon$ small enough and any $w \in B^{\perp}_{\delta_0}$, we have:
		\begin{equation*}
			\begin{aligned}
				\|S(w)\|_{*}& \le C\| - \pi^{\perp}F(u) - \pi^{\perp}N(w)\|_{**},\\
				&\le C\|F(u)\|_{**} + C\delta_0^2\\
				&\le C\epsilon^2 + C\delta_0^2 \le \delta_0.
			\end{aligned}
		\end{equation*}
		Therefore $S(w)$ is also in $B^{\perp}_{\delta_0}$.
		Hence we get that for $w, w' \in B^{\perp}_{\delta_0}$
		\begin{equation*}
			\begin{aligned}
				\|S(w)- S(w')\|_{*} &= \|\mathbb L^{-1}\pi^{\perp}N(w) - \mathbb L^{-1}\pi^{\perp}N(w')\|_{**}\\
				&\le \frac{1}{2}\|w - w'\|_{*}.
			\end{aligned}
		\end{equation*}
		Thus the map $S$ is a contraction map and has a unique fixed point in $B^{\perp}_{\delta_0}$.
		Finally we will estimate the fixed point $w$:
		\begin{equation*}
			\begin{aligned}
				\|w\|_{*} &\le \|S(0)\|_{*} + \|S(0) - S(w)\|_{*}\\
				&\le C\|F(u)\|_{**} + \frac{1}{2}\|w\|_{*},\\
				\Rightarrow & \|w\|_{*}\le C\epsilon^2.
			\end{aligned}
		\end{equation*}
	\end{proof}
	
	In the last of the section, we show the Lipschitz continuity of $w$ with respect to $f$ and $g$.
	
	\begin{theorem}
		If $(f_j, g_j), j =1,2$ are two pairs of functions satisfying
		\begin{equation*}
			\|(f_j, g_j)\|_{2,1, q} \le C,
		\end{equation*}
		then the two solutions $w_j$ to \eqref{nonlinear equ} with respect to $(f_j, g_j)$ in Theorem \ref{thm of existence of the nonlinear problem} satisfy:
		\begin{equation*}
			\|w_1 - w_2\|_{*} \le C\epsilon^2\|(f_1, g_1) - (f_2, g_2)\|_{2, 1, q}.
		\end{equation*}
	\end{theorem}
	\begin{proof}
		Denote $u_j$ is the approximate solution with respect to $(f_j, g_j)$ and rewrite \eqref{nonlinear equ}:
		\begin{equation*}
			\pi^{\perp}[F(u_j) + \mathbb L_jw_j + N_j(w_j)] = 0.
		\end{equation*}
		
		 So we denote $S_j(w) : \mathbb L^{-1}_j[-\pi^{\perp}F(u_j) - \pi^{\perp}N_j(w_j)]$. Then we only need to estimate $F(u_1) - F(u_2)$.
		
		In fact, by the results in section 1, we get :
		\begin{equation*}
			\|F(u_1) - F(u_2)\|_{**}\le C\epsilon^2\|(f_1, g_1) - (f_2, g_2)\|_{2, 1, q}.
		\end{equation*}
		Thus we have:
		\begin{equation*}
			\begin{aligned}
				\|w_1 - w_2\|_{*} &= \|S_1(w_1) - S_2(w_2)\|_{*}\\
				&\le \|S_1(w_1) - S_1(w_2)\|_{*} + \|S_1(w_2) - S_2(w_2)\|_{*}\\
				&\le \frac{1}{2}\|w_1 - w_2\|_{*} + C\|F(u_1) - F(u_2)\|_{**}.\\
			\end{aligned}
		\end{equation*}
		Therefore we get:
		\begin{equation*}
			\|w_1 - w_2\|_{*} \le C\epsilon^2\|(f_1, g_1) - (f_2, g_2)\|_{2, 1, q}.
		\end{equation*}
	\end{proof}

	\section{From Jacobi operator to the true solution}
	In the section, we hope to find appropriate $f$ and $g$, which satisfies \eqref{disturbance}, such that the equation holds:
	\begin{equation*}
		F(u) + \mathbb L w + N(w) = 0.
	\end{equation*}
	This is involved with Jacobi operator, which is discussed in Arezzo-Pacard \cite{Arezzo}. Becasue of the decay of $f$ and $g$, we can get similar results for the Jacobi operator in dimension 4.

	Multiply $T_j$ and integral on $\{a, b\}$, then using the results in section 1 we get:
	\begin{equation}\label{Jacobi equ1}
		\begin{aligned}
			&\rho^{-6}\partial_s^2f + 2\rho^{-4}\cos2s\cdot\partial_s f + \rho^{-2}\cdot\partial_\theta^2f - 2\rho^{-2}\cos2s\cdot\partial_\theta g + 2\rho^{-6}\cdot f \\
			&- \rho^{-2}(\cos2s)^2\cdot f  = F_1,\\
			&\rho^{-6}\partial_s^2g + 2\rho^{-4}\cos2s\cdot\partial_s g + \rho^{-2}\cdot\partial_\theta^2g + 2\rho^{-2}\cos2s\cdot\partial_\theta f +2\rho^{-6}\cdot g \\
			&- \rho^{-2}(\cos2s)^2\cdot g = F_2,
		\end{aligned}
	\end{equation}
	where
	\begin{equation*}
		\|F_j\|_{0, 2, q} \le C.
	\end{equation*}
	
	If we denote $N = ie^{-is}f\cdot\Theta + ie^{is}g\cdot\Theta^{\perp}$, then we can rewrite \eqref{Jacobi equ1}:
	\begin{equation*}
		L_HN = ie^{-is}F_1\cdot\Theta + ie^{is}F_2\cdot\Theta^{\perp}.
	\end{equation*}
	
	Next we will recall the results of Arezzo-Pacard in \cite{Arezzo}. Define the variable $t$ by:
	\begin{equation*}
		e^{-2t} = \frac{\sin2s}{1 - \cos2s} = \frac{\cos s}{\sin s}.
	\end{equation*}
	Then we have:
	\begin{equation*}
		\begin{aligned}
			dt = \frac{1}{\sin2s}ds,\text{ } \sin2s = (\cosh2t)^{-1},\\
			\cos2s = -\tanh2t.
		\end{aligned}
	\end{equation*}
	We define the conjugate operator:
	\begin{equation*}
		\mathcal L_H : = (\sin2s)^{-1}L_H,
	\end{equation*}
	then we have:
	\begin{equation*}
		\begin{aligned}
			\mathcal L_HN = \partial_t^2(f, g) + \partial_\theta^2(f, g) + 2\tanh2t\cdot\partial_\theta(g, -f) + [\frac{3}{(\cosh2t)^2} - 1](f, g).
		\end{aligned}
	\end{equation*}
	
	Let $t$ tend to $-\infty$, we get that $\mathcal L_H$ is asymptotic to the following differential operator:
	\begin{equation*}
		\mathcal L_0N := \partial_t^2(f, g) + \partial_\theta^2(f, g) - 2\partial_\theta(g, -f) - (f, g).
	\end{equation*}
	
	On $S^1$, the spectrum of $\Delta^0$ and $\Delta^1$ are given by:
	\begin{equation*}
		\begin{aligned}
			\sigma(\Delta^0) = \{k^2: k\ge 0\},\\
			\sigma(\Delta^1) = \{k^2: k \ge  0\}.
		\end{aligned}
	\end{equation*}
	
	We remark that we consider $g\cdot\Theta^{\perp}$ as the tangent vector field on $S^1$. Then the corresponding 1-form is $g$d$\theta$.
	
	We define
	\begin{equation*}
		\mathcal V^0: = \{(f, g): df = 0, dg = 0\}.
	\end{equation*}
	For all $k\ge 1$, we define:
	\begin{equation*}
		\mathcal V^k := \{(f, g):\Delta(f, g) = -k^2(f, g)\}.
	\end{equation*}
	In fact, we know the eigenfunctions of $\Delta^0$ corresponding to eigenvalue $k^2$ are $\cos k\theta$ and $\sin k\theta$. So we can always assume that:
	\begin{equation*}
		\begin{aligned}
			f &= f_0 + \sum\limits_{k \ge 1}f_{k,1}\cos k\theta + \sum\limits_{k \ge 1}f_{k,2}\sin k\theta,\\
			g &= g_0 + \sum\limits_{k \ge 1}g_{k,1}\cos k\theta + \sum\limits_{k \ge 1}g_{k,2}\sin k\theta. \\
		\end{aligned}
	\end{equation*}
	This time the equation $\mathcal L_H(f, g) = 0$ yields the couple system:

	\begin{numcases}{}
		\ddot{f}_{k,1} - k^2f_{k,1} + 2k\tanh 2t\cdot g_{k,2} + [\frac{3}{(\cosh2t)^2} - 1]f_{k, 1} = 0,\\
			\ddot g_{k, 2} - k^2 g_{k, 2} + 2k\tanh 2t\cdot f_{k, 1} + [\frac{3}{(\cosh2t)^2} - 1]g_{k, 2} = 0,
	\end{numcases}

	when $k \neq 0$. If we denote $h_{k, 1} = f_{k, 1} + g_{k, 2}$, $h_{k, 2} = f_{k, 1} - g_{k, 2}$, then we have:
	\begin{numcases}{}
		 \label{ode1}\ddot{h}_{k,1} - k^2h_{k,1} + 2k\tanh 2t\cdot h_{k,1} + [\frac{3}{(\cosh2t)^2} - 1]h_{k,1} = 0,\\
		\label{ode2} \ddot{h}_{k,2} - k^2f_{k,2} - 2k\tanh 2t\cdot h_{k,2} + [\frac{3}{(\cosh2t)^2} - 1]h_{k,2} = 0.
	\end{numcases}
	 And for $k = 0$ ,we get:
	\begin{equation*}
		\begin{cases}
			\ddot f_0 - [1 - \frac{3}{(\cosh2t)^2}]f_0 = 0,\\
			\ddot g_0 - [1 - \frac{3}{(\cosh2t)^2}]g_0 = 0.
		\end{cases}
	\end{equation*}
	
	We find that when $k = 0$, the asymptotic behavior of $f_0$, $g_0$ at both $\pm\infty $ is governed by the exponents
	\begin{equation*}
		\mu_0^{\pm} = \pm 1.
	\end{equation*}
	And when $k > 0$, the asymptotic behavior of $f_{k, j}$, $g_{k, j}$ at both $\pm\infty $ is governed by the following sets of indicial roots:
	\begin{equation*}
		\mu^{\pm}_k = \pm (k + 1), \nu^{\pm}_k = \pm(k - 1).
	\end{equation*}
	
	In \cite{Arezzo}, Arezzo-Pacard introduced all the geometric Jacobi fields as well as their explicit expression.
	
	\begin{itemize}
		\item 1. Jacobi fields corresponding to translation:
		\begin{equation*}
			f = (a\cdot\Theta)\sin(\alpha + s), g = (a\cdot\Theta^{\perp}) \sin(\alpha - s),
		\end{equation*}
		where $\alpha \in \mathbb R$ and $a \in \mathbb R^2$
		\item 2. Jacobi fields corresponding to dilations:
		\begin{equation*}
			f = (\sin2s)^{\frac{1}{2}},  g = 0.
		\end{equation*}
		\item 3. Jacobi fields corresponding to the action of $SU(2)$ :
		\begin{equation*}
			\begin{aligned}
				f = (\sin2s)^{-\frac{1}{2}}\cos2s (A\Theta\cdot\Theta),\\
				g = (\sin2s)^{-\frac{1}{2}}(A\Theta\cdot\Theta^{\perp}),
			\end{aligned}
		\end{equation*}
		where $A$ is a $2\times2$ symmetric matrix.
		\item 4. Jacobi fields corresponding to the action of $O(4)/SU(2)$
		\begin{equation*}
			f = 0, g = (\sin2s)^{\frac{1}{2}}(A\Theta\cdot\Theta^{\perp}),
		\end{equation*}
		or
		\begin{equation*}
			f = 0, g = (\sin2s)^{-\frac{1}{2}}\cos2s\cdot(A\Theta\cdot\Theta^{\perp}),
		\end{equation*}
		where $A$ is a $2 \times 2$ skew-symmetric matrix.
	\end{itemize}	
	
	We observed that the Jacobi field 2 and 4 are in $\mathcal V^0$, and 1 is in $\mathcal V^1$. As for the Jacobi field 3, its decay is bad. Besides, similarly in Proposition 11.2 and 11.3 in \cite{Arezzo}, we get:
	\begin{proposition}
		There exists $\hat t_0 > 0$ such that if $(f, g)$ is a solution of $\mathcal L(f, g) = 0$ in $(t_1, t_2)\times\mathbb S^1$ , for some $t_1 < -\hat t_0$ and $t_2 > \hat t_0$ (with $U = 0$ on the boundary if either $t_1 > -\infty$ or $t_2 < \infty$, satisfying
		\begin{equation*}
			\|(f, g)\|_{0, \delta, q} < 1 ,
		\end{equation*}
		for some $ 0 < \delta < \frac{1}{4}$ and with the property that for any $t \in (t_1, t_2)$, ($f, g)(t, \cdot)$ is orthogonal to $\mathcal V^0$ in the $L^2$ sense on $\mathbb S^1$, then $(f, g)= 0$.
	\end{proposition}
	The proof is the same as Proposition 11.3 in \cite{Arezzo}.
	In fact we can drop the condition for $\hat t_0$, since the equations can be reduced to \eqref{ode1} and \eqref{ode2}.
	Finally we can get the following two results:
	\begin{proposition}\label{Schauder1}
		Assume $0 < \delta < \frac{1}{4}$ is fixed. Then for all $t_0 \in \mathbb R$, there exists an operator
		\begin{equation*}
			\mathcal G_{t_0}: W^{0, q}_\delta([t_0, +\infty) \times \mathbb S^1; \mathbb R\times \mathbb R) \rightarrow W^{2, q}_\delta([t_0, +\infty) \times \mathbb S^1; \mathbb R\times \mathbb R),
		\end{equation*}
		such that for all $V := (F_1, F_2) \in W^{0, q}_\delta([t_0, +\infty) \times \mathbb S^1; \mathbb R\times \mathbb R)$, if for all $t > t_0$, $V(t, \cdot)$ is orthogonal to $\mathcal V^0$ in the $L^2$ sense on $\mathbb S^1$, then $U = \mathcal G_{t_0} V$ is the unique solution of
		\begin{equation*}
			\begin{cases}
				\mathcal L_HU = V, \text{ in }[t_0, +\infty) \times \mathbb S^1,\\
				V = 0, \text{ on } \{t_0\}\times\mathbb S^1,
			\end{cases}
		\end{equation*}
		which belongs to $W^{2, q}_\delta([t_0, +\infty) \times \mathbb S^1; \mathbb R\times \mathbb R)$. Furthermore, $\|U\|_{2, \delta, q} \le C \|V\|_{0, \delta, q}$.
	\end{proposition}
	
	\begin{proposition}\label{Schauder2}
		Assume that $\delta > 1$ is fixed. There exists some const $c>0$ and an operator
		\begin{equation*}
			\mathcal G_{t_0}: W^{0, q}_\delta([t_0, +\infty) \times \mathbb S^1; \mathbb R\times \mathbb R) \rightarrow W^{2, q}_\delta([t_0, +\infty) \times \mathbb S^1; \mathbb R\times \mathbb R),
		\end{equation*}
		such that for all $V : = (F_1, F_2) \in W^{0, q}_\delta([t_0, +\infty) \times \mathbb S^1; \mathbb R\times \mathbb R)$, if for all $t > t_0$, $V$ is constant on $\mathbb S^1$, then $U = \mathcal G_{t_0} V$ is the unique solution of
		\begin{equation*}
			\begin{cases}
				\mathcal L_HU = V, \text{ in }[t_0, +\infty) \times \mathbb S^1,\\
				V = 0, \text{ on } \{t_0\}\times\mathbb S^1,
			\end{cases}
		\end{equation*}
		which belongs to $W^{2, q}_\delta([t_0, +\infty) \times \mathbb S^1; \mathbb R\times \mathbb R)$. Furthermore, $\|U\|_{2, \delta, q} \le C \|V\|_{0, \delta, q}$.
	\end{proposition}
	
	With these results, we can solve the last Jacobi operator equations \eqref{Jacobi equ1}:
	\begin{theorem}\label{the solution to Jacobi operator}
		When $\epsilon$ is small enough, there exists a solution to \eqref{Jacobi equ1}, satisfying:
		\begin{equation*}
			\|(f, g)\|_{2, 1, q} \le C.
		\end{equation*}
	\end{theorem}
	\begin{proof}
		By the previous calculation, we can rewrite$(F_1, F_2)$:
		\begin{equation*}
			(F_1, F_2) = (F_1', F_2') + (F_1'', F_2''),
		\end{equation*}
		\begin{equation*}
			\begin{aligned}
				\epsilon^3F_1' = &\int_{\tilde r} \pi\epsilon^3 \frac{(V')^2}{\tilde r}\cdot( 6(\cos2s)^2\rho^{-5}\tilde r^2 - 4\rho^{-9}\tilde r^2 )\\
			& + \int_{\tilde r} 4\pi \tilde rV(-V'' + \frac{V'}{\tilde r})(\cos2s)^2\epsilon^3\rho^{-5}- \int_{\tilde r}4\pi \tilde rV \epsilon^3\rho^{-9}(-V'' + \frac{V'}{\tilde r})\\
			&+\int_{\tilde r} (\cos2s)^2\epsilon^3\rho^{-5}\cdot 4\pi \tilde r^2VUU'- \int_{\tilde r} (\cos2s)^2\epsilon^3\rho^{-5}\cdot 2\pi \tilde r^2V^2UU',
			\end{aligned}
		\end{equation*}
		In fact $(F_1', F_2')$ comes from the proof of Theorem \ref{Jacobi} and does not depend on $(f, g)$.
		Besides, we know:
		\begin{equation*}
			F_1''  = \epsilon^{-3}\int_{r, \phi}<\mathbb Lw + N(w) , T_1>.
		\end{equation*}
		If we suppose $\|(f, g)\|_{2, 1, q} \le \epsilon^{-\frac{1}{2}}$ and define the operator $J(f, g) := \mathcal L_H^{-1}(F_1, F_2)$, then $\|(F_1'', F_2'')\|_{0, 1, q} \le  C\epsilon^{\frac{1}{2}}$ and by Proposition \ref{Schauder1} and \ref{Schauder2}, we have:
		\begin{equation*}
			\|J(f, g)\|_{2, 1, q} \le C \|(F_1, F_2)\|_{0, 1, q}\le C \le \epsilon^{-\frac{1}{2}}.
		\end{equation*}
		So $J$ maps from $B_{\epsilon^{-\frac{1}{2}}}$ to $B_{\epsilon^{-\frac{1}{2}}}$ in $W^{2, q}_\delta([t_0, +\infty) \times \mathbb S^1; \mathbb R\times \mathbb R)$.
		
		Similarly in the proof of Theorem \ref{thm of existence of the nonlinear problem}, we get that $J$ is a contraction map and has a fixed point. And we can also get:
		\begin{equation*}
			\|(f, g)\|_{2, 1, q} \le C.
		\end{equation*}
	\end{proof}
	
	In the end, we will explain that we have got the desired solution when $f, g$ are the ones in Theorem \ref{the solution to Jacobi operator}: (see \cite{Brendle})
	\begin{theorem}
		If $w$ satisfies that:
		\begin{equation*}
			F(u) + \mathbb Lw + N(w) = 0,
		\end{equation*}
		then $F(u + w) = 0$. In other words, $v := u + w$ is the solution to Ginzburg Landau equations.
	\end{theorem}
	\begin{proof}
		We only need to show that $TT^*(B, \eta) = 0$.
		For short, we define $\xi = T^*(B, \eta)$ and $v = (\tilde A, \tilde \psi) + u + w$. Recall that:
		\begin{equation*}
			\begin{aligned}
				T\xi = (d\xi, i\xi\tilde\psi),\\
				T^*(B, \eta) = d^*B - Im(\bar\eta\tilde\psi).
			\end{aligned}
		\end{equation*}

		 Then by the definition of $F(v) = 0$, we get:
		\begin{equation*}
			\begin{aligned}
				-\Delta_{\tilde A}\tilde\psi + \frac{\lambda}{2}(|\tilde\psi|^2-1)\tilde\psi + i\xi\tilde\psi = 0,\\
		d^*d\tilde A - Im(\nabla_{\tilde A}\tilde\psi\cdot\overline{\tilde\psi})  + d\xi = 0.
			\end{aligned}
		\end{equation*}
		Therefore we have:
		\begin{equation*}
			\begin{aligned}
				-\Delta\xi &= d^*d\xi = d^*Im(\nabla_{\tilde A}\tilde \psi\cdot\overline{\tilde\psi})\\
				&= Im(\nabla_{\tilde A}^*\nabla_{\tilde A}\tilde \psi\cdot\overline{\tilde\psi})\\
				&=Im(-i\xi\tilde \psi\cdot\overline{\tilde\psi}) = -\xi|\tilde\psi|^2.
			\end{aligned}
		\end{equation*}
		So we conclude that $\xi = 0$.
	\end{proof}
		
	\begin{appendices}

		\section{The proof in section 3}

\subsection{The proof of Lemma \ref{lemma2}}

\begin{proof}
	
$d^2 = \inf\limits_\rho \frac12\rho^2 + \frac12\rho^{-2} + 2\rho_1^2 - \sqrt2\rho_1\sqrt{\rho^2 + \rho^{-2} + 2\cos\alpha} $.

Suppose $t = \sqrt{\rho^2 + \rho^{-2} + 2\cos\alpha}$, then we only need to consider the function $h(t) = \frac12t^2 - \sqrt2\rho_1t + 2\rho_1^2 - \cos\alpha$, $t \geqslant \sqrt{2 + 2\cos\alpha}$.

So when $\rho_1 > \sqrt{1 + \cos\alpha}$, we have
\begin{equation*}
	\inf_t h(t) = h(\sqrt2\rho_1)
	\Rightarrow 2\rho_1^2 = \rho^2 + \rho^{-2} + 2\cos\alpha.
\end{equation*}
At this case, there are two closest points $\rho', \rho''$, satisfying  $\rho'\rho'' = 1$.

When $\rho_1 \leqslant \sqrt{1 + \cos\alpha}$, we have $\inf\limits_t h(t) = h(\sqrt{2 + 2\cos\alpha})$, then $\rho = 1$. So there is only one closest point. This completes the proof.
\end{proof}

\subsection{The proof of Lemma \ref{lemma3}}

\begin{proof}
If $\rho_1 \neq \rho_2$, then $s \neq \frac\pi4$ and the conclusion naturally holds by Lemma \ref{lemma1} . Now
we have $a^2 + b^2 < \frac1{\epsilon^2\sin2s}$.

If $\rho_1 = \rho_2$ and there is only one closest point, then by symmetry we know $s = \frac{\pi}4$. As a result,
\begin{equation*}
	\begin{aligned}
	    &y = \left(\frac{\Theta}{\sqrt2\epsilon} + \frac{a\Theta}{\sqrt{2}} - \frac{b\Theta^\perp}{\sqrt2}, \frac{\Theta}{\sqrt2\epsilon} + \frac{a\Theta}{\sqrt2} + \frac{b\Theta^\perp}{\sqrt2}\right),\\
		&\rho_1^2 = \frac{(a + \epsilon^{-1})^2 + b^2}2, \text{ } e^{i(\theta_1 - \theta_2)} = \frac{\epsilon^{-1} + a -bi}{\epsilon^{-1} + a +bi}.
	\end{aligned}
\end{equation*}

So by $\rho_1 < \frac{\sqrt{2}}{2\epsilon}\sqrt{2 + 2\cos(\theta_1 - \theta_2)}$ from Lemma \ref{lemma2} , we have
\begin{equation}
	\left(\left(a + \epsilon^{-1}\right)^2 + b^2\right)^2 < \frac4{\epsilon^2}\left(a + \epsilon^{-1}\right)^2
	\Longrightarrow (a + \epsilon^{-1})^2 + b^2 < 2\epsilon^{-1}\left\vert a + \epsilon^{-1}\right\vert.
\end{equation}

If $a > -\epsilon^{-1}$, then (1)$\Rightarrow a^2 + b^2 < \epsilon^{-2}$.

If $a < -\epsilon^{-1}$, then (1)$\Rightarrow (a + \epsilon^{-1})^2 + b^2 < -2\epsilon^{-1}(a + \epsilon^{-1})$

$\qquad\Rightarrow (a + 2\epsilon^{-1})^2 + b^2 < \epsilon^{-2}$.
At this case, we observe that
\begin{equation*}
	 y = \left(- \frac{\Theta}{\sqrt2\epsilon} + \frac{\left(a + \frac2\epsilon\right)\Theta}{\sqrt{2}} - \frac{b\Theta^\perp}{\sqrt2}, - \frac{\Theta}{\sqrt2\epsilon} + \frac{\left(a + \frac2\epsilon\right)\epsilon}{\sqrt2} + \frac{b\Theta^\perp}{\sqrt2}\right).
\end{equation*}
This means the two different cases about $a$ in fact are the same.

In conclusion, when $a^2 + b^2 < \frac1{\epsilon^2\sin2s}$, then $y = T(s, \theta, a, b)$ has a unique nearest point to $\Gamma_\epsilon$.
\end{proof}

\subsection{The proof of Lemma \ref{lemma4}}

\begin{proof}
If $\rho_1 = \rho_2$, then from the previous arguments in Lemma \ref{lemma2} , after rescaling, we know
\begin{equation*}
	\begin{aligned}
		dist(y, \Gamma_\epsilon) &= \epsilon^{-2} + 2\rho_1^2 - 2\rho_1\epsilon^{-1}\sqrt{1 + \cos\alpha}\\
		&= a^2 + b^2.
	\end{aligned}
\end{equation*}

If $\rho_1 > \rho_2$, then $0 < s < \frac\pi4$. From Lemma \ref{lemma1}, we know the closest point in $\Gamma_\epsilon$ to y is in $\{\rho_1 > \rho_2\}$. Next, we will show that there is only one extreme point  of distance in $\{\rho_1 > \rho_2\} \cap \Gamma_\epsilon$. For simplicity, we only need to consider $\Gamma$ rather than $\Gamma_\epsilon$. Denote $d = dist(y, \Gamma)$.
\begin{equation*}
    \begin{aligned}
	     &f(\rho, \theta) =\frac{1}{2} \rho^2 +\frac{1}{2} \rho^{-2} + \rho_1^2 + \rho_2^2 - \sqrt2\rho_1\rho\cos(\theta - \theta_1) - \sqrt2\rho_2\rho^{-1}\cos(\theta - \theta_2),\\
	     &d^2 = \inf_{\rho, \theta} f(\rho, \theta).
	\end{aligned}
\end{equation*}

For  $\forall \rho$, $\exists \theta_0$, we have
\begin{equation*}
    \begin{aligned}
	     &\frac{\partial f}{\partial\theta}(\rho, \theta_0) = \frac{\partial f}{\partial \theta}(\rho, \theta_0 + \pi) = 0,\\
	     &f(\rho, \theta_0) =\frac{1}{2} \rho^2 + \frac12\rho^{-2} + \rho_1^2 + \rho_2^2 - \sqrt2\sqrt{\rho_1^2\rho^2 + \rho_2^2\rho^{-2} + 2\rho_1\rho_2\cos(\theta_1 - \theta_2)},\\
	     &f(\rho, \theta_0 + \pi) = \frac12\rho^2 + \frac12\rho^{-2} + \rho_1^2 + \rho_2^2 + \sqrt2\sqrt{\rho_1^2\rho^2 + \rho_2^2\rho^{-2} + 2\rho_1\rho_2\cos(\theta_1 - \theta_2)}.
	\end{aligned}
\end{equation*}
So we know that $f(\rho, \theta_0 + \pi)$ is strictly increasing in $[1, +\infty)$, which means there is no extreme point.

On the other hand, by Lemma \ref{lemma1}, if $r_2 > r_1 > 1$ satisfy $\frac{\partial f}{\partial\rho}(r_1, \theta_0) = \frac{\partial f}{\partial \rho}(r_2, \theta_0) = 0$, then  $f(r_1, \theta_0) < f(r_2, \theta_0)$. So there is only one local minimum point in $\{\rho_1 > \rho_2\}$, and it is also a global minimum point. By rescaling, the conclusion also holds for $\Gamma_\epsilon$. By calculation, we know $\left(\frac{\cos s}{\epsilon\sqrt{\sin2s}}\Theta, \frac{\sin s}{\epsilon \sqrt{\sin2s}}\Theta\right)$ is the minimum point. So we have $dist(y, X_\epsilon) = \sqrt{a^2 + b^2}$. This completes the proof.
\end{proof}

		\section{The proof in section 4}
		The metric matrix of Fermi coordinates is $$(g_{ij}) =
\begin{pmatrix}
	\tilde{a}_{11} & \tilde{a}_{12} & 0 & 0\\
	\tilde{a}_{21} & \tilde{a}_{22} & -b\cos2s & a\cos2s\\
	0 & -b\cos2s & 1 & 0\\
	0 & a\cos2s & 0 & 1\\
\end{pmatrix},$$
and $$ \tilde A = (\tilde a_{ij}) =
\begin{pmatrix}
	\left(a - \frac1{\epsilon(\sin2s)^{\frac32}}\right)^2 + b^2 & \frac{-2b}{\epsilon\sqrt{\sin2s}}\\
	\frac{-2b}{\epsilon\sqrt{\sin2s}} & a^2 + b^2 + \frac1{\epsilon^2\sin2s} + \frac{2a\sqrt{\sin2s}}\epsilon\\
\end{pmatrix}.$$
	
	The inverse matrix of $(g_{ij})$ is:
 $$(g^{ij}) =
 \begin{pmatrix}
 	\frac1c + \frac{4b^2}{c^2e\epsilon^2\sin2s}&\frac{2b}{ce\epsilon\sqrt{\sin2s}}&\frac{2b^2\cos2s}{ce\epsilon\sqrt{\sin2s}}&\frac{-2ab\cos2s}{ce\epsilon\sqrt{\sin2s}}\\
 	\frac{2b}{ce\epsilon\sqrt{\sin2s}}&\frac1e&\frac{b\cos2s}e&\frac{-a\cos2s}e\\
 	\frac{2b^2\cos2s}{ce\epsilon\sqrt{\sin2s}}&\frac{b\cos2s}e&1 + \frac{b^2\cos^22s}e&\frac{-ab\cos^22s}e\\
 	\frac{-2ab\cos2s}{ce\epsilon\sqrt{\sin2s}}&\frac{-a\cos2s}e&\frac{-ab\cos^22s}e&1 + \frac{a^2\cos^22s}e\\
 \end{pmatrix}.
 $$

 By direct calculation, we have:
\begin{equation*}
	\begin{aligned}
		&\frac{\partial_s(\sqrt G)}{\sqrt G} = -4(\sin2s)^{-1}\cos2s- 6(a^2 + b^2)\epsilon^2\cdot(\sin2s)^2\cos2s + O(\epsilon^4\rho^{-6})\\
		&=-4(\sin2s)^{-1}\cos2s + o(\epsilon^{2-\sigma}\rho^{-2}),\\
		&\frac{\partial_a(\sqrt G)}{\sqrt G} = -2a\epsilon^2(\sin2s)^3 + O(\epsilon^4\rho^{-9}),\\
		&\frac{\partial_b(\sqrt G)}{\sqrt G} = -2b\epsilon^2(\sin2s)^3 + O(\epsilon^4\rho^{-9}),\\
		&\frac1c = \epsilon^2\sin^32s + 2a\epsilon^3(\sin2s)^\frac92 + (3a^2-b^2)\epsilon^4(\sin2s)^6 +o(\epsilon^{5-\sigma}\rho^{-12}),\\
		&\partial_s \left(\frac1c\right) = 6\epsilon^2\sin^22s\cos2s + 18a\epsilon^3(\sin2s)^\frac72\cos2s + O(\epsilon^4\rho^{-10})\\
		&\partial_a \left(\frac1c\right) = 2\epsilon^3(\sin2s)^\frac92 + O(\epsilon^4\rho^{-12}),\\
		&\partial_b \left(\frac1c\right) = O(\epsilon^4\rho^{-12}),\\
		&\frac1e = \epsilon^2\sin2s - 2a\epsilon^3(\sin2s)^\frac52 + 3(a^2 + b^2)\epsilon^4(\sin2s)^4 + o(\epsilon^{5-\sigma}\rho^{-8}),\\
		&\partial_s \left(\frac1e\right) = 2\epsilon^2\cos2s - 10a\epsilon^3(\sin2s)^\frac32\cos2s + O(\epsilon^4\rho^{-6}),\\
		&\partial_a \left(\frac1e\right) = -2\epsilon^3(\sin2s)^\frac52 + O(\epsilon^4\rho^{-8}),\\
		&\partial_b \left(\frac1e\right) = O(\epsilon^4\rho^{-8}).
	\end{aligned}
\end{equation*}

	Using the formula: if $\alpha = \alpha_{i_1...i_p}dx^{i_1}\wedge...\wedge dx^{i_p} $, then
	\begin{equation*}
		(d^*\alpha)_{i_1...i_{p-1}} = -g^{kl}(\frac{\partial\alpha_{ki_1...i_{p-1}}}{\partial x^l} - \Gamma^j_{kl}\alpha_{ji_1...i_{p-1}}),
	\end{equation*}
	with
	\begin{equation*}
		\Gamma^j_{kl} = \frac12g^{ji}(g_{ki,l} + g_{li,k} - g_{kl,i}),
	\end{equation*}
	
	\begin{equation*}
		\begin{aligned}
			A =A_\tau dy^\tau,
		\end{aligned}
	\end{equation*}
	with $y^1 = s, y^2 = \theta, y^3 = a, y^4 = b$
	
	So we have:
	\begin{equation*}
		\begin{aligned}
			d^*A &= -g^{kl}(\frac{\partial A_k}{\partial y^l} - \Gamma^j_{kl}A_j).\\
		\end{aligned}
	\end{equation*}

	If we define $d^*dA = (d^*dA)_\tau dy^\tau$, then
	\begin{equation*}
		(d^*dA)_\tau = -g^{kl}(\frac{\partial A_{k\tau}}{\partial y^l} - \Gamma^j_{kl}A_{j\tau}).
	\end{equation*}
	The most complicated item is $B^j := g^{kl}\Gamma^j_{kl}$.

	On the other hand, we have another definition for $d^*$:
	
	If $\alpha = \alpha _{i_1...i_r} dx^{i_1}\wedge...\wedge dx^{i_r} $, then
	\begin{equation*}
		\begin{aligned}
			*\alpha &= \eta_{i_1...i_r j_{r+1}...j_n} \alpha^{i_1...i_r}dx^{j_{r+1}} \wedge...\wedge dx^{j_n},\\
			d^*\alpha &= (-1)^{n(r+1) +1}*d*,
		\end{aligned}
	\end{equation*}	
	with
	\begin{equation*}
		\begin{aligned}
			\eta_{i_1...i_n} = \sqrt G\delta^{1...n}_{i_1...i_n},\\
			\alpha^{i_1...i_r} = g^{i_1k_1}...g^{i_rk_r}\alpha_{k_1...k_r}.
		\end{aligned}
	\end{equation*}
	
	Therefore, for $A = A_ldy^l$ we have:
	\begin{equation*}
		\begin{aligned}
			d^*A &= -\partial_jA^j - \frac{\partial_j(\sqrt G)}{\sqrt G}A^j\\
			& = (-\partial_jg^{jl} - \frac{\partial_j(\sqrt G)}{\sqrt G}\cdot g^{jl})A_l - g^{kl}\partial_kA_l.
		\end{aligned}
	\end{equation*}
	
	As a result, we get:
	\begin{equation*}
		B^j = -\partial_kg^{kj} - \frac{\partial_k(\sqrt G)}{\sqrt G}\cdot g^{jk}.
	\end{equation*}

	First, we have the following estimates for $(g^{ij})$ when it is close to the minimal surface:

	\begin{equation*}
		\begin{aligned}
			g^{11} &= \epsilon^2\rho^{-6} + 2a\epsilon^3\rho^{-9} + 3(a^2+b^2)\epsilon^4\rho^{-12} + o(\epsilon^{5-\sigma}\rho^{-12}),\\
			g^{12} &= 2b\epsilon^{3}\rho^{-7} + o(\epsilon^{5-\sigma}\rho^{-10}),\\
			g^{13}&= 2b^2\cos2s\cdot\epsilon^3\rho^{-7} +o(\epsilon^4\rho^{-9}),\\
			g^{14} &= -2ab\cos2s\cdot\epsilon^3\rho^{-7} + o(\epsilon^4\rho^{-9}),\\
			g^{22} &=\epsilon^2\rho^{-2}-2a\epsilon^3\rho^{-5} + 3(a^2 + b^2)\epsilon^4\rho^{-8} + o(\epsilon^{5-\sigma}\rho^{-8}),\\
			 g^{23} &= b\cos2s\cdot\epsilon^2\rho^{-2} -2ab\cos2s\cdot\epsilon^3\rho^{-5} + o(\epsilon^{4-\sigma}\rho^{-7}),\\
			 g^{24} &=-a\cos2s\cdot\epsilon^2\rho^{-2} +2a^2\cos2s\cdot\epsilon^3\rho^{-5} + o(\epsilon^{4-\sigma}\rho^{-7}),\\
			 g^{33} &= 1+b^2(\cos2s)^2\cdot\epsilon^2\rho^{-2} -2ab^2(\cos2s)^2\cdot\epsilon^3\rho^{-5} + o(\epsilon^{4-\sigma}\rho^{-4}),\\
			 g^{34} &=-ab(\cos2s)^2\cdot\epsilon^2\rho^{-2} +2a^2b(\cos2s)^2\cdot\epsilon^3\rho^{-5}   + o(\epsilon^{4-\sigma}\rho^{-4}),\\
			 g^{44} &=1+a^2(\cos2s)^2\cdot\epsilon^2\rho^{-2} - 2a^3(\cos2s)^2\cdot\epsilon^3\rho^{-5} +o(\epsilon^{4-\sigma}\rho^{-4}).
    	\end{aligned}
	\end{equation*}

	Besides we also have:
	\begin{equation*}
		\begin{aligned}
			B^1 &= -2\epsilon^2\rho^{-4}\cos2s - 8a\epsilon^3\rho^{-7}\cos2s + O(\epsilon^{4}\rho^{-8}),\\
			B^2 &=-4b\epsilon^3\rho^{-5}\cos2s + O(\epsilon^{4}\rho^{-7}),\\
			B^3 &= 2a\epsilon^2\rho^{-6}  + a(\cos2s)^2\epsilon^2\rho^{-2}   - (2a^2 + 4b^2)(\cos2s)^2\epsilon^3\rho^{-5} + 4b^2\epsilon^3\rho^{-9} + O(\epsilon^{4}\rho^{-7}),\\
			B^4 &= 2b\epsilon^2\rho^{-6} + b(\cos2s)^2\epsilon^2\rho^{-2}  + 2ab(\cos2s)^2\epsilon^3\rho^{-5} - 4ab\epsilon^3\rho^{-9} + O(\epsilon^{4}\rho^{-7}).\\
		\end{aligned}
	\end{equation*}

\subsection{The proof of Lemma \ref{lemma4.1}}
\begin{proof}
	Denote $w_{ij} = \frac1{\sqrt G} \partial_i\left(\sqrt G g^{ij}\partial_j\right)$ and we only consider the $O(\epsilon^3)$ items because of the exponential decay.

\begin{equation*}
	\begin{aligned}
		&w_{1,1} = \left[2\epsilon^2(\sin2s)^2\cos2s + 10a\epsilon^3(\sin2s)^\frac72\cos2s + O(\epsilon^4)\right]\partial_s +\\
		&\left[\epsilon^2(\sin2s)^3 + 2a\epsilon^3(\sin2s)^\frac92 + O(\epsilon^4)\right]\partial_s\partial_s\\
		&w_{2,1} = \left[2b\epsilon^3(\sin2s)^\frac72 + O(\epsilon^4)\right]\partial_s\partial_\theta,\\
		&w_{3,1} = O(\epsilon^4)\partial_s + \left[2b^2\epsilon^3(\sin2s)^\frac72\cos2s + O(\epsilon^4)\right]\partial_s\partial_a,\\
		&w_{4,1} = \left[-2a\epsilon^3(\sin2s)^\frac72\cos2s + O(\epsilon^4)\right]\partial_s + \left[-2ab\epsilon^3(\sin2s)^\frac72\cos2s + O(\epsilon^4)\right]\partial_s\partial_b.
	\end{aligned}
\end{equation*}

\begin{equation*}
	\begin{aligned}
		&w_{1,2} = \left[6b\epsilon^3(\sin2s)^\frac52\cos2s + O(\epsilon^4)\right]\partial_\theta + \left[2b\epsilon^3(\sin2s)^\frac72 + O(\epsilon^4)\right]\partial_s\partial_\theta,\\
		&w_{2,2} = \left[\epsilon^2\sin2s - 2a\epsilon^3(\sin2s)^\frac52 + O(\epsilon^4)\right]\partial_\theta\partial_\theta,\\
		&w_{3,2} = \left[-2b\epsilon^3(\sin2s)^\frac52\cos2s + O(\epsilon^4)\right]\partial_\theta + \\
		&\left[b\epsilon^2\sin2s\cos2s - 2ab\epsilon^3(\sin2s)^\frac52\cos2s + O(\epsilon^4)\right]\partial_a\partial_\theta,\\
		&w_{4,2} = O(\epsilon^4)\partial_\theta + \left[-a\epsilon^2\sin2s\cos2s + 2a^2\epsilon^3(\sin2s)^\frac52\cos2s + O(\epsilon^4)\right]\partial_b\partial_\theta.\\
	\end{aligned}
\end{equation*}

\begin{equation*}
	\begin{aligned}
		&w_{1,3} = \left[6b^2\epsilon^3(\sin2s)^\frac52(\cos2s)^2 - 4b^2\epsilon^3(\sin2s)^\frac92 + O(\epsilon^4)\right]\partial_a \\&+ \left[2b^2\epsilon^3(\sin2s)^\frac72\cos2s + O(\epsilon^4)\right]\partial_s\partial_a,\\
		&w_{2,3} = b\cos2s\left[\epsilon^2\sin2s - 2a\epsilon^3(\sin2s)^\frac52 + O(\epsilon^4)\right]\partial_\theta\partial_a,\\
		&w_{3,3} = \left[-2a\epsilon^2(\sin2s)^3 -2b^2\epsilon^3(\sin2s)^\frac52(\cos2s)^2 + O(\epsilon^4)\right]\partial_a + \\
		&\left[1 + b^2\epsilon^2\sin2s(\cos2s)^2 - 2ab^2\epsilon^3(\sin2s)^\frac52(\cos2s)^2 + O(\epsilon^4)\right]\partial_a\partial_a,\\
		&w_{4,3} = \left[-a\epsilon^2\sin2s(\cos2s)^2 + 2a^2\epsilon^3(\sin2s)^\frac52(\cos2s)^2 + O(\epsilon^4)\right]\partial_a +\\
		&\left[-ab\epsilon^2\sin2s(\cos2s)^2 + 2a^2b\epsilon^3(\sin2s)^\frac52(\cos2s)^2 + O(\epsilon^4)\right]\partial_a\partial_b.\\
	\end{aligned}
\end{equation*}

\begin{equation*}
\begin{aligned}
		&w_{1,4} = \left[- 6ab\epsilon^3(\sin2s)^\frac52(\cos2s)^2 + 4ab\epsilon^3(\sin2s)^\frac92 + O(\epsilon^4)\right]\partial_b \\&+ \left[-2ab\epsilon^3(\sin2s)^\frac72\cos2s + O(\epsilon^4)\right]\partial_s\partial_b,\\
		&w_{2,4} = -a\cos2s\left[\epsilon^2\sin2s - 2a\epsilon^3(\sin2s)^\frac52 + O(\epsilon^4)\right]\partial_\theta\partial_b,\\
		&w_{3,4} = \left[-b\epsilon^2\sin2s(\cos2s)^2 + 4ab\epsilon^3(\sin2s)^\frac52(\cos2s)^2 + O(\epsilon^4)\right]\partial_b +\\
		&\left[-ab\epsilon^2\sin2s(\cos2s)^2 + 2a^2b\epsilon^3(\sin2s)^\frac52(\cos2s)^2 + O(\epsilon^4)\right]\partial_a\partial_b,\\
		&w_{4,4} = \left[-2b\epsilon^2(\sin2s)^3 + O(\epsilon^4)\right]\partial_b + \\
		&\left[1 + a^2\epsilon^2\sin2s(\cos2s)^2 - 2a^3\epsilon^3(\sin2s)^\frac52(\cos2s)^2 + O(\epsilon^4)\right]\partial_b\partial_b.\\
	\end{aligned}
\end{equation*}
The we can get the conclusion.
\end{proof}

	\subsection{The proof of Lemma \ref{lemma4.2}}
	\begin{proof}
		\begin{equation*}
		\begin{aligned}
			d^*A &= -g^{kl}\frac{\partial A_k}{\partial y^l} + B^jA_j\\
			&= -g^{1l}\frac{\partial A_1}{\partial y^l} -g^{2l}\frac{\partial A_2}{\partial y^l}-g^{3l}\frac{\partial A_3}{\partial y^l} -g^{4l}\frac{\partial A_4}{\partial y^l} + B^1A_1 + B^3A_3 + B^4A_4\\
			&= \epsilon^2\rho^{-6}\cdot(\epsilon f_{ss}\cdot A_3 + \epsilon g_{ss}\cdot A_4) + \epsilon^2\rho^{-2}\cdot(\epsilon f_{\theta\theta} A_3 + \epsilon g_{\theta\theta}A_4) \\
			&+2b\cos2s\cdot\epsilon^3\rho^{-2}f_\theta\partial_3A_3 + b\cos2s\cdot\epsilon^3\rho^{-2}g_\theta(\partial_4A_3 +\partial_3A_4) \\
			&-a\cos2s\cdot\epsilon^3\rho^{-2}f_\theta(\partial_3A_4 +\partial_4A_3) - 2a\cos2s\cdot\epsilon^3\rho^{-2}g_\theta\partial_4A_4\\
			&-(-ab(\cos2s)^2\epsilon^2\rho^{-2} + 2a^2b(\cos2s)^2\epsilon^3\rho^{-5})(\partial_4A_3+\partial
			_3A_4)\\
			& -(1+b^2(\cos2s)^2\epsilon^2\rho^{-2} - 2ab^2(\cos2s)^2\epsilon^3\rho^{-5})\partial_3A_3\\
			& - (1+a^2(\cos2s)^2\epsilon^2\rho^{-2}-2a^3(\cos2s)^2\epsilon^3\rho^{-5})\partial_4A_4\\
			& + 2\epsilon^2\rho^{-4}\cos2s\cdot(\epsilon f_sA_3 + \epsilon g_sA_4)\\
			&+(2a\epsilon^2\rho^{-6}  + a(\cos2s)^2\epsilon^2\rho^{-2}   - (2a^2 + 4b^2)(\cos2s)^2\epsilon^3\rho^{-5} + 4b^2\epsilon^3\rho^{-9})A_3\\
			& +(2b\epsilon^2\rho^{-6} + b(\cos2s)^2\epsilon^2\rho^{-2}  + 2ab(\cos2s)^2\epsilon^3\rho^{-5} - 4ab\epsilon^3\rho^{-9})A_4 + o(\epsilon^{4}).\\
		\end{aligned}
	\end{equation*}
	\end{proof}

	\subsection{The proof of Lemma \ref{lemma4.3}}
	\begin{proof}
		\begin{equation*}
		\begin{aligned}
			<A, d\psi> &=  -2b\cos2s\cdot\epsilon^3\rho^{-2}f_\theta A_3\partial_3\psi - b\cos2s\cdot\epsilon^3\rho^{-2}g_\theta (A_3\partial_4\psi + A_4\partial_3\psi) \\
			&+a\cos2s\cdot\epsilon^{3}\rho^{-2}f_\theta (A_4\partial_3\psi + A_3\partial_4\psi) + 2a\cos2s\cdot\epsilon^{3}\rho^{-2}g_\theta A_4\partial_4\psi\\
			&+(1+b^2(\cos2s)^2\cdot\epsilon^2\rho^{-2} -2ab^2(\cos2s)^2\cdot\epsilon^3\rho^{-5})A_3\cdot\partial_3\psi \\
			&+ (-ab(\cos2s)^2\cdot\epsilon^2\rho^{-2} +2a^2b(\cos2s)^2\cdot\epsilon^3\rho^{-5})(A_3\partial_4\psi + A_4\partial_3\psi) +\\
			& +(1+a^2(\cos2s)^2\cdot\epsilon^2\rho^{-2} - 2a^3(\cos2s)^2\cdot\epsilon^3\rho^{-5})A_4\partial_4\psi\\
			&  + o(\epsilon^{4}).\\
		\end{aligned}
	\end{equation*}
	\end{proof}

	\subsection{The proof of Lemma \ref{lemma4.4}}
	\begin{proof}
		\begin{equation*}
		\begin{aligned}			
			(d^*dA)_1 &= -\partial_3A_{31} - \partial_4A_{41} + O(\epsilon^{2})	\\
			&= \cos\tilde\phi \cdot\epsilon g_s(\frac{V''}{\tilde r} - \frac{V'}{\tilde r^2}) - \sin\tilde\phi\cdot\epsilon f_s(\frac{V''}{\tilde r} - \frac{V'}{\tilde r^2})\\
			&+O(\epsilon^2),\\
			(d^*dA)_2 &= -\partial_3A_{32} - \partial_4A_{42} + O(\epsilon^{2})	\\
			& = \cos\tilde\phi \cdot\epsilon g_\theta(\frac{V''}{r} - \frac{V'}{r^2}) - \sin\tilde\phi\cdot\epsilon f_\theta(\frac{V''}{r} - \frac{V'}{r^2})\\
			&+O(\epsilon^2),\\
		\end{aligned}
	\end{equation*}
	
	\begin{equation*}
		\begin{aligned}
			(d^*dA)_3 &= -\epsilon^2\rho^{-6}\partial_1A_{13}
			-\epsilon^2\rho^{-2}\partial_2A_{23} - b\cos2s\cdot\epsilon^2\rho^{-2} \partial_3A_{23}\\
			&+a\cos2s\cdot\epsilon^2\rho^{-2}\partial_4A_{23}   +a\cos2s\cdot\epsilon^2\rho^{-2}\partial_2A_{43} \\
			&-(-ab(\cos2s)^2\epsilon^2\rho^{-2}  + 2a^2b(\cos2s)^2\epsilon^3\rho^{-5} )\partial_3A_{43}\\
			&-(1+a^2(\cos2s)^2\epsilon^2\rho^{-2} - 2a^3(\cos2s)^2\epsilon^3\rho^{-5} )\partial_4A_{43} \\
			&-2\epsilon^2\rho^{-4}\cos2s \cdot A_{13} + O(\epsilon^4)\\
			&+(2b\epsilon^2\rho^{-6} + b(\cos2s)^2\epsilon^2\rho^{-2}  + 2ab(\cos2s)^2\epsilon^3\rho^{-5} - 4ab\epsilon^3\rho^{-9})A_{43},\\
		\end{aligned}
	\end{equation*}
	
	\begin{equation*}
		\begin{aligned}
			(d^*dA)_4 &= -\epsilon^2\rho^{-6}\cdot\partial_1A_{14}
			 -\epsilon^2\rho^{-2}\partial_2A_{24} - b\cos2s\cdot\epsilon^2\rho^{-2}\partial_3A_{24}\\
			&+a\cos2s\cdot\epsilon^2\rho^{-2}\partial_4A_{24}    -b\cos2s\cdot\epsilon^2\rho^{-2}\partial_2A_{34}\\
			&-(1+b^2(\cos2s)^2\cdot\epsilon^2\rho^{-2} -2ab^2(\cos2s)^2\cdot\epsilon^3\rho^{-5})\partial_3A_{34}\\
			&-(-ab(\cos2s)^2\cdot\epsilon^2\rho^{-2} +2a^2b(\cos2s)^2\cdot\epsilon^3\rho^{-5})\partial_4A_{34}\\
			&+ (-2\epsilon^2\rho^{-4}\cos2s )A_{14}  +O(\epsilon^4)  \\
			&+(2a\epsilon^2\rho^{-6}  + a(\cos2s)^2\epsilon^2\rho^{-2}   - (2a^2 + 4b^2)(\cos2s)^2\epsilon^3\rho^{-5} + 4b^2\epsilon^3\rho^{-9})A_{34}.\\
		\end{aligned}
	\end{equation*}
	\end{proof}

			\subsection{The proof of Lemma \ref{lemma4.5}}

			\begin{proof}
				\begin{equation*}
		\begin{aligned}
			Im(\nabla_A\psi\cdot\bar{\psi}) &= (\psi_1 \partial_1\psi_2 - \psi_2\partial_1\psi_1 - A_1|\psi|^2)dy^1 \\
			&+ (\psi_1\partial_2\psi_2 - \psi_2\partial_2\psi_1 - A_2|\psi|^2)dy^2\\
			&+(-\frac{U^2}{\tilde r}\sin\tilde\phi - A_3|\psi|^2)dy^3
			+(\frac{U^2}{\tilde r}\cos\tilde\phi - A_4|\psi|^2)dy^4\\
			&= (\epsilon f_s \cdot \frac{U^2 - VU^2}{\tilde r}\sin\tilde\phi - \epsilon g_s\cdot\frac{U^2 - VU^2}{\tilde r}\cos\tilde\phi)dy^1\\
			&+ (\epsilon f_\theta \cdot \frac{U^2 - VU^2}{\tilde r}\sin\tilde\phi - \epsilon g_\theta\cdot\frac{U^2- VU^2}{\tilde r}\cos\tilde\phi)dy^2\\
			&+(-\frac{U^2}{\tilde r}\sin\tilde\phi - A_3|\psi|^2)dy^3
			+(\frac{U^2}{\tilde r}\cos\tilde\phi - A_4|\psi|^2)dy^4.\\
		\end{aligned}
	\end{equation*}	
			\end{proof}

		\section{The proof of lemma \ref{the boundary condition}}	
	\begin{proof}

	Recall that
	\begin{equation*}
	\begin{aligned}
		\omega(x) &= -2\pi\int_\Gamma *\left[d_x\left(G(x-y)\right)\wedge\chi(y)\right]dH^2(y)\\
		&= \frac{\pi}{4\omega_4}\int_{\Gamma}i_{\frac{d_x|x-y|^2}{|x-y|^4}}(*\chi(y))dH^2(y).
	\end{aligned}
	\end{equation*}
	Denote $x = (\rho_1e^{i\theta_1}, \rho_2e^{i\theta_2})$, $y = (\frac{\sqrt2}{2}\rho e^{i\theta}, \frac{\sqrt2}{2}\rho^{-1}e^{i\theta})$, $\alpha = \theta_2 - \theta_1$. Let $e_1 = \frac{1}{\sqrt{1+\rho^{-4}}}(\partial_{\rho_1} - \rho^{-2}\partial_{\rho_2})$, $e_2 = \frac{\sqrt2}{\sqrt{\rho^2 + \rho^{-2}}}(\partial_{\theta_1} + \partial_{\theta_2})$, which are induced by $\partial_\rho$ and $\partial_\theta$. Then by direct calculation, when $\rho_1 = \rho_2$, we have:
	
	1) If $\sqrt{1+\cos\alpha}\le \rho_1$, then
	\begin{equation*}
		\begin{aligned}
			&\tilde \rho(x)^2 + \tilde \rho(x)^{-2} + 2\cos\alpha = 2\rho_1^2, \text{ where $(\frac{\sqrt2}{2}\tilde\rho e^{i\tilde\theta}, \frac{\sqrt2}{2}\tilde\rho^{-1}e^{i\tilde\theta})$ is the closest points in $\Gamma$ to $x$}. \\
			&a(x) = \cos\alpha\sqrt{\rho_1^2-\cos\alpha}, \text{ }b(x) = \sin\alpha\sqrt{\rho_1^2 - \cos\alpha},\\
			&dist(x, \Gamma) = \sqrt{\rho_1^2 - \cos\alpha},\\
			&d\phi = -d\theta_1 + d\theta_2.\\
		\end{aligned}
	\end{equation*}
	
	2) If $\sqrt{1+\cos\alpha}> \rho_1$, then
	\begin{equation*}
		\begin{aligned}
			a(x) = \rho_1\sqrt{1+\cos\alpha} -1, \text{ } b(x) = \frac{\rho_1\sin\alpha}{\sqrt{1+\cos\alpha}}.
		\end{aligned}
	\end{equation*}
	so $d\phi = -j(\alpha)d\theta_1 + j(\alpha)d\theta_2$, where $j(\alpha)$ is a smooth function.		
	\begin{equation*}
		\begin{aligned}
			e^1 = \frac{1}{\sqrt{1+\rho^{-4}}}(d\rho_1 - \rho^{-2}d\rho_2), \ e^2 = \frac{1}{\sqrt{2\rho^2 + 2\rho^{-2}}}(\rho^2d\theta_1 + \rho^{-2}d\theta_2), \\
			(1+\rho^{-4})\chi = \frac{1}{\sqrt2}\rho d\rho_1\wedge d\theta_1 + \frac{1}{\sqrt2}\rho^{-3}d\rho_1\wedge d\theta_2 - \frac{1}{\sqrt2}\rho^{-1}d\rho_2\wedge d\theta_1 - \frac{1}{\sqrt2}\rho^{-5}d\rho_2\wedge d\theta_2.
		\end{aligned}
	\end{equation*}		
	The metric matrix with respect to $(\rho_1,\theta_1,\rho_2,\theta_2)$ is :
	
	$$\begin{pmatrix}
		1 &0&0&0\\
		0&\rho_1^2&0&0\\
		0&0&1&0\\
		0&0&0&\rho_2^2
	\end{pmatrix},$$
	
	\begin{equation*}
		\begin{aligned}
			*(d\rho_1\wedge d\theta_1) & = \rho_1^{-1}\rho_2d\rho_2\wedge d\theta_2,\\
			*(d\rho_1\wedge d\theta_2) &= -\rho_1\rho^{-1}_2d\rho_2\wedge d\theta_1,\\
			*(d\rho_2 \wedge d\theta_1) &= -\rho_1^{-1}\rho_2d\rho_1\wedge d\theta_2,\\
			*(d\rho_2\wedge d\theta_2) &= \rho_1\rho_2^{-1}d\rho_1\wedge d\theta_1.
		\end{aligned}
	\end{equation*}
	So we can get :
	\begin{equation*}
		\begin{aligned}
			(1 + \rho^{-4})*\chi &= \frac{1}{\sqrt2}\rho^{-1}d\rho_2\wedge d\theta_2 - \frac{1}{\sqrt2}\rho^{-1}d\rho_2\wedge d\theta_1\\
			&+ \frac{1}{\sqrt2}\rho^{-3}d\rho_1\wedge d\theta_2 - \frac{1}{\sqrt2}\rho^{-3} d\rho_1\wedge d\theta_1,
		\end{aligned}
	\end{equation*}

	\begin{equation*}
		|x-y|^2 = \rho_1^2 + \rho_2^2 + \frac{1}{2}\rho^2 + \frac{1}{2}\rho^{-2} - \sqrt{2}\rho_1\rho\cos(\theta - \theta_1) - \sqrt{2}\rho_2\rho^{-1}\cos(\theta - \theta_2).
	\end{equation*}
	
	In $\Gamma$, we know:
	\begin{equation*}
		\begin{aligned}
			\partial_\rho &= \partial_{\rho_1} - \rho^{-2}\partial_{\rho_2},\\
			\partial_\theta &= \partial_{\theta_1} + \partial_{\theta_2}.\\
		\end{aligned}
	\end{equation*}

We remark that, we denote $d\rho_1$ as $d\rho_1(y)$ for simplicity. Similarly $d\theta_1$, $d\rho_2$, $d\theta_2$ are the same.

\begin{equation*}
		\begin{aligned}
			d\rho_1(x) &= \cos(\theta - \theta_1)d\rho_1 - \sin(\theta - \theta_1)\frac{\sqrt{2}}{2}\rho d\theta_1,\\
			\rho_1(x)d\theta_1(x) & = \sin(\theta - \theta_1)d\rho_1 + \cos(\theta - \theta_1)\frac{\sqrt{2}}{2}\rho d\theta_1,\\
			d\rho_2(x) &= \cos(\theta - \theta_2)d\rho_2 - \sin(\theta - \theta_2)\frac{\sqrt{2}}{2}\rho^{-1} d\theta_2,\\
			\rho_2(x)d\theta_2(x) & = \sin(\theta - \theta_2)d\rho_2 + \cos(\theta - \theta_2)\frac{\sqrt{2}}{2}\rho^{-1} d\theta_2.\\
		\end{aligned}
	\end{equation*}
	
	Conversely,
	\begin{equation*}
		\begin{aligned}
			d\rho_1 &= \cos(\theta - \theta_1)d\rho_1(x) + \rho_1\sin(\theta - \theta_1)d\theta_1(x),\\
			\rho d\theta_1 &= \sqrt2\rho_1\cos(\theta - \theta_1)d\theta_1(x) - \sqrt2\sin(\theta - \theta_1)d\rho_1(x),\\
			d\rho_2 &= \cos(\theta - \theta_2)d\rho_2(x) + \rho_2\sin(\theta - \theta_2)d\theta_2(x),\\
			\rho^{-1}d\theta_2 &= \sqrt2\rho_2\cos(\theta - \theta_2)d\theta_2(x) - \sqrt2\sin(\theta - \theta_2)d\rho_2(x),
		\end{aligned}
	\end{equation*}
			
	\begin{equation*}
		\begin{aligned}
			d_x|x-y|^2 &= [2\rho_1 - \sqrt2 \rho\cos(\theta - \theta_1)]d\rho_1(x) - \sqrt2\rho_1\rho\sin(\theta - \theta_1)d\theta_1(x)\\
			& + [2\rho_2 - \sqrt2 \rho^{-1}\cos(\theta - \theta_2) ]d\rho_2(x) - \sqrt2 \rho_2\rho^{-1}\sin(\theta - \theta_2)d\theta_2(x)\\
			&= [2\rho_1 - \sqrt2 \rho\cos(\theta - \theta_1)]\cdot[\cos(\theta - \theta_1)d\rho_1 - \sin(\theta - \theta_1)\cdot \frac{\sqrt2}{2}\rho d\theta_1 ]\\
			&-\sqrt2\rho\sin(\theta - \theta_1)\cdot[\sin(\theta - \theta_1)d\rho_1 + \cos(\theta - \theta_1)\cdot\frac{\sqrt2}{2}\rho d\theta_1  ]\\
			&+ [2\rho_2 - \sqrt2 \rho^{-1}\cos(\theta - \theta_2) ]\cdot[\cos(\theta - \theta_2)d\rho_2 - \sin(\theta - \theta_2)\cdot \frac{\sqrt2}{2}\rho^{-1} d\theta_2]\\
			&- \sqrt2 \rho^{-1}\sin(\theta - \theta_2)\cdot [\sin(\theta - \theta_2)d\rho_2 + \cos(\theta - \theta_2)\cdot \frac{\sqrt2}{2}\rho^{-1} d\theta_2]\\
			& = 2\rho_1 \cos(\theta - \theta_1)d\rho_1 - \sin(\theta - \theta_1)\cdot \sqrt2\rho\rho_1 d\theta_1 - \sqrt2\rho d\rho_1\\
			&+ 2\rho_2 \cos(\theta - \theta_2)d\rho_2 - \sin(\theta - \theta_2)\cdot \sqrt2\rho^{-1}\rho_2 d\theta_2 - \sqrt2\rho^{-1}d\rho_2,
		\end{aligned}
	\end{equation*}

	\begin{equation*}
		\begin{aligned}
			&(1 + \rho^{-4})i_{d_x|x-y|^2}(*\chi) = [2\rho_1 \cos(\theta - \theta_1) - \sqrt2\rho ]\cdot (\frac{1}{\sqrt2}\rho^{-3}d\theta_2 - \frac{1}{\sqrt2}\rho^{-3}d\theta_1)\\
			& - \sin(\theta - \theta_1)\cdot \sqrt2\rho\rho_1 [\sqrt{2}\rho^{-3}d\rho_2 + \sqrt{2}\rho^{-5}d\rho_1]\\
			&+[2\rho_2 \cos(\theta - \theta_2) - \sqrt2\rho^{-1}] \cdot(\frac{1}{\sqrt2}\rho^{-1}d\theta_2 - \frac{1}{\sqrt2}\rho^{-1}d\theta_1)\\
			& - \sin(\theta - \theta_2)\cdot \sqrt2\rho^{-1}\rho_2 \cdot[-\sqrt{2}\rho d\rho_2 -  \sqrt{2}\rho^{-1}d\rho_1  ]\\
			&= [2\rho^{-2}\rho_2\sin(\theta - \theta_2) - 2\rho^{-4}\rho_1\sin(\theta - \theta_1)]d\rho_1 \\
			&+ [2\rho_2\sin(\theta - \theta_2) - 2\rho^{-2}\rho_1\sin(\theta - \theta_1)]d\rho_2\\
			& + [2\rho^{-2} - \sqrt{2}\rho_1\rho^{-3}\cos(\theta - \theta_1)  - \sqrt{2}\rho^{-1}\rho_2\cos(\theta - \theta_2)]d\theta_1\\
			&+ [-2\rho^{-2} + \sqrt{2}\rho^{-3}\rho_1\cos(\theta - \theta_1) + \sqrt{2}\rho^{-1}\rho_2\cos(\theta - \theta_2) ]d\theta_2\\
			&= [2\rho^{-2}\rho_2\sin(\theta - \theta_2) - 2\rho^{-4}\rho_1\sin(\theta - \theta_1)]\cdot[\cos(\theta - \theta_1)d\rho_1(x) + \rho_1\sin(\theta - \theta_1)d\theta_1(x)] \\
			&+ [2\rho_2\sin(\theta - \theta_2) - 2\rho^{-2}\rho_1\sin(\theta - \theta_1)]\cdot[\cos(\theta - \theta_2)d\rho_2(x) + \rho_2\sin(\theta - \theta_2)d\theta_2(x)]\\
			& + [2\rho^{-2} - \sqrt{2}\rho_1\rho^{-3}\cos(\theta - \theta_1)  - \sqrt{2}\rho^{-1}\rho_2\cos(\theta - \theta_2)]\cdot\\
			&[\sqrt2\rho^{-1}\rho_1\cos(\theta - \theta_1)d\theta_1(x) - \sqrt2\rho^{-1}\sin(\theta - \theta_1)d\rho_1(x)]\\
			&+ [-2\rho^{-2} + \sqrt{2}\rho^{-3}\rho_1\cos(\theta - \theta_1) + \sqrt{2}\rho^{-1}\rho_2\cos(\theta - \theta_2) ]\cdot\\
			&[\sqrt2\rho\rho_2\cos(\theta - \theta_2)d\theta_2(x) - \sqrt2\rho\sin(\theta - \theta_2)d\rho_2(x)]\\
			&= [-2\sqrt{2}\rho^{-3}\sin(\theta - \theta_1) + 2\rho^{-2}\rho_2\sin(2\theta - \theta_1 - \theta_2)]d\rho_1(x)\\
			&+ [2\sqrt{2}\rho^{-1}\sin(\theta - \theta_2) - 2\rho^{-2}\rho_1\sin(2\theta - \theta_1 - \theta_2)]d\rho_2(x)\\
			&+ [-2\rho^{-4}\rho_1^2 + 2\sqrt{2}\rho^{-3}\rho_1\cos(\theta - \theta_1) - 2\rho^{-2}\rho_1\rho_2\cos(2\theta - \theta_1 - \theta_2)]d\theta_1(x)\\
			&+ [2\rho_2^2 - 2\sqrt{2}\rho^{-1}\rho_2\cos(\theta - \theta_2) + 2\rho^{-2}\rho_1\rho_2\cos(2\theta - \theta_1 - \theta_2)]d\theta_2(x).
		\end{aligned}
	\end{equation*}

	Besides, for the manifold $\Gamma$, the metric matrix with respect to $(\rho, \theta)$ is:
	$$\begin{pmatrix}
		\frac{1}{2} + \frac{1}{2}\rho^{-4}&0\\
		0&\frac{1}{2}\rho^2 + \frac{1}{2}\rho^{-2}
	\end{pmatrix}.$$
	
	Therefore, we finally get:
	\begin{equation*}
		\begin{aligned}
			\omega(x) &= \frac{\pi}{4\omega_4}\int_{\Gamma}i_{\frac{d_x|x-y|^2}{|x-y|^4}}(*\chi(y))dH^2(y)\\
			&= \frac{\pi}{8\omega_4}\int_{\rho, \theta} [\rho_1^2 + \rho^2_2 + \frac{1}{2}\rho^2 + \frac{1}{2}\rho^{-2} - \sqrt{2}\rho_1\rho\cos(\theta - \theta_1) - \sqrt{2}\rho_2\rho^{-1}\cos(\theta - \theta_2)]^{-2}\cdot\\
			&  \{[-2\sqrt{2}\rho^{-3}\sin(\theta - \theta_1) + 2\rho^{-2}\rho_2\sin(2\theta - \theta_1 - \theta_2)]d\rho_1(x)\\
			&+ [2\sqrt{2}\rho^{-1}\sin(\theta - \theta_2) - 2\rho^{-2}\rho_1\sin(2\theta - \theta_1 - \theta_2)]d\rho_2(x)\\
			&+ [-2\rho^{-4}\rho_1^2 + 2\sqrt{2}\rho^{-3}\rho_1\cos(\theta - \theta_1) - 2\rho^{-2}\rho_1\rho_2\cos(2\theta - \theta_1 - \theta_2)]d\theta_1(x)\\
			&+ [2\rho_2^2 - 2\sqrt{2}\rho^{-1}\rho_2\cos(\theta - \theta_2) + 2\rho^{-2}\rho_1\rho_2\cos(2\theta - \theta_1 - \theta_2)]d\theta_2(x)\}\cdot \rho d\rho d\theta\\
			&=\frac{\pi}{8\omega_4}\int_{\rho, \theta} |x-y|^{-4}\cdot
			 \{[-2\sqrt{2}\rho^{-2}\sin(\theta - \theta_1) + 2\rho^{-1}\rho_2\sin(2\theta - \theta_1 - \theta_2)]d\rho_1(x)\\
			&+ [2\sqrt{2}\sin(\theta - \theta_2) - 2\rho^{-1}\rho_1\sin(2\theta - \theta_1 - \theta_2)]d\rho_2(x)\\
			&+ [-2\rho^{-3}\rho_1^2 + 2\sqrt{2}\rho^{-2}\rho_1\cos(\theta - \theta_1) - 2\rho^{-1}\rho_1\rho_2\cos(2\theta - \theta_1 - \theta_2)]d\theta_1(x)\\
			&+ [2\rho_2^2\rho - 2\sqrt{2}\rho_2\cos(\theta - \theta_2) + 2\rho^{-1}\rho_1\rho_2\cos(2\theta - \theta_1 - \theta_2)]d\theta_2(x)\}d\rho d\theta.
		\end{aligned}
	\end{equation*}
	
	If we change $\theta$ into $-\theta + \theta_1 + \theta_2$, then:
	\begin{equation*}
		\begin{aligned}
			\theta - \theta_1 \rightarrow -\theta + \theta_2,\\
			\theta - \theta_2 \rightarrow -\theta + \theta_1.
		\end{aligned}
	\end{equation*}
	
	Thus, if $\rho_1 = \rho_2$, i.e. $x \in \{(\rho_1, \theta_1, \rho_2, \theta_2 ):\rho_1 = \rho_2\}$, we have
	
	\begin{equation*}
		\begin{aligned}
			\omega(x) &= \frac{\pi}{8\omega_4}\int_{\rho \ge 1, 0\le\theta < 2\pi} |x-y|^{-4}\cdot
			 \{[-2\sqrt{2}\rho^{-2}\sin(\theta - \theta_1) + 4\rho^{-1}\rho_1\sin(2\theta - \theta_1 - \theta_2) \\
			 &- 2\sqrt2\sin(\theta - \theta_2)  ]d\rho_1(x)\\
			&+ [2\sqrt{2}\sin(\theta - \theta_2) - 4\rho^{-1}\rho_1\sin(2\theta - \theta_1 - \theta_2) +2\sqrt{2}\rho^{-2}\sin(\theta - \theta_1) ]d\rho_2(x)\\
			&+ [-2\rho^{-3}\rho_1^2 + 2\sqrt{2}\rho^{-2}\rho_1\cos(\theta - \theta_1)  + 2\rho_1^2\rho - 2\sqrt{2}\rho_1\cos(\theta - \theta_2)  ]d\theta_1(x)\\
			&+ [2\rho_1^2\rho - 2\sqrt{2}\rho_1\cos(\theta - \theta_2) - 2\rho_1^2\rho^{-3} + 2\sqrt{2}\rho_1\rho^{-2}\cos(\theta - \theta_1) ]d\theta_2(x)\}d\rho d\theta.
		\end{aligned}
	\end{equation*}

	Note that the tangent space of $\{\rho_1 = \rho_2\}$ is spanned by $\partial_{\rho_1} + \partial_{\rho_2}$, $\partial_{\theta_1}$, $\partial_{\theta_2}$. So we get that:
	
	\begin{equation*}
		\begin{aligned}
			&\omega_\top(x)  = \frac{\pi}{8\omega_4}\int_{\rho \ge 1, 0\le\theta < 2\pi} |x-y|^{-4}\cdot
			 \{[-2\rho^{-3}\rho_1^2 + 2\sqrt{2}\rho^{-2}\rho_1\cos(\theta - \theta_1)  \\
			 &+ 2\rho_1^2\rho - 2\sqrt{2}\rho_1\cos(\theta - \theta_2)  ]d\theta_1(x)\\
			&+ [2\rho_1^2\rho - 2\sqrt{2}\rho_1\cos(\theta - \theta_2) - 2\rho_1^2\rho^{-3} + 2\sqrt{2}\rho_1\rho^{-2}\cos(\theta - \theta_1) ]d\theta_2(x)\}d\rho d\theta.
		\end{aligned}
	\end{equation*}
	
	\end{proof}

	\end{appendices}

\end{document}